\documentclass{amsart} \usepackage{latexsym,amsxtra,amscd,ifthen} \usepackage{amsfonts} \usepackage{verbatim}
\usepackage{amsmath} \usepackage{amsthm} \usepackage{amssymb}

\input xy \xyoption{matrix} \xyoption{arrow}\xyoption{frame} 
 
 \newcommand{\edge}{\ar@{-}}

\numberwithin{equation}{section}

\theoremstyle{plain} \newtheorem{theorem}{Theorem}[section] \newtheorem{lemma}[theorem]{Lemma}
\newtheorem{proposition}[theorem]{Proposition} \newtheorem{corollary}[theorem]{Corollary}

\theoremstyle{definition} \newtheorem{definition}[theorem]{Definition} \newtheorem{example}[theorem]{Example}
\newtheorem{examples}[theorem]{Examples}
\newtheorem{rexamples}[theorem]{Remarks and Examples}

 \newtheorem{remarks}[theorem]{Remarks} \newtheorem*{remark*}{Remark} \newtheorem{question}[theorem]{Question}

\newcommand{\gnoc}{\mathrel{{\lower.2ex\hbox{$\backsim$}}\llap{\raise.45ex\hbox{=}}}}

\begin{document}

\title[The Cohen Macaulay property] {The Cohen Macaulay property for noncommutative rings}

\author[K.A. Brown]{K.A. Brown} \address{School of Mathematics and Statistics\\ University of Glasgow\\ Glasgow G12 8QW\\
Scotland.} \email{Ken.Brown@glasgow.ac.uk}

\author[M.J. MacLeod]{M.J. MacLeod}

\begin{abstract} Let $R$ be a noetherian ring which is a finite module over its centre $Z(R)$. This paper studies the consequences for $R$ of the hypothesis that it is a maximal Cohen Macaulay $Z(R)$-module. Old results are reviewed and a number of new results are proved. The additional hypothesis of homological grade symmetry is proposed as the appropriate extra lever needed to extend the classical commutative homological hierarchy to this setting, and results are given offering evidence in support of this proposal.

\end{abstract}

\maketitle

\section{Introduction} \label{intro}
\subsection{} \label{intro1} Let $R$ be a noetherian ring which is a finitely generated module over its centre $Z(R)$. We shall call $R$ $Z(R)$-\emph{Macaulay} if it is a maximal Cohen-Macaulay $Z(R)$-module, or, more generally, $C$-\emph{Macaulay} where $Z(R)$ is replaced by a central subring $C$. Rings with this property have been much studied over the past 40 years, starting with \cite{V}, continuing through, for example, \cite{BHM}, \cite{Lev}, \cite{GN}, to a host of more recent works. This ubiquity stems largely from the fact that many classes of noetherian PI rings of central concern in noncommutative algebra and representation theory satisfy the Macaulay condition - see Examples \ref{CMex} and $\S$5 for a far from complete list. This paper has several purposes:
\begin{enumerate}
\item[(i)] to review the basic properties of $Z(R)$-Macaulay rings;
\item[(ii)] to prove a number of new results concerning these rings;
\item[(iii)] to explain how to strengthen the $Z(R)$-Macaulay hypothesis so as to recover in this noncommutative setting the familiar homological hierarchy from commutative noetherian ring theory (regular$\Rightarrow$Gorenstein$\Rightarrow$Cohen-Macaulay, with reverse implications valid when relevant homological dimensions are finite);
\item[(iv)] to use this PI setting as a testing ground for speculation about how the Cohen-Macaulay hypothesis should be generalised in noncommutative algebra.
\end{enumerate}
Aims (i) and (ii) are covered in $\S\S$2,3 and reviewed in subsection \ref{intro2}, and aims (iii) and (iv) in $\S\S$4,5, reviewed in subsection \ref{intro3}.

\subsection{}\label{intro2} Call the ring $R$ \emph{equicodimensional} if every maximal ideal of $R$ has the same height; this condition has proved to be an important (if sometimes erroneously overlooked) hypothesis in applications of the $Z(R)$-Macaulay property. Regarding aim (i), we review the fundamental yoga of $C$-sequences and the interactions of the equicodimensionality and Krull homogeneity conditions, and examine the dependence of the $C$-Macaulay property on the choice of the central subring $C$, ($\S\S$\ref{basic}, \ref{equidim}, \ref{depend}). We also recall an important result of Goto and Nishida \cite{GN} on catenarity of $Z(R)$-Macaulay rings, Theorem \ref{chains}. Concerning (ii), the main new structural results obtained are as follows:
\begin{itemize}
  \item  Let $R$ be a $C$-Macaulay ring which is an affine $k$-algebra, $k$ a field. Then $R$ and $C$ are finite direct sums, $R = \oplus_i R_i \supseteq \oplus_i C_i = C$, with $R_i$ and $C_i$ equicodimensional and $R_i$ a $C_i$-Macaulay ring for all $i$ (Theorem \ref{affinesplit}).
  \item  Let $R$ be equicodimensional and $Z(R)$-Macaulay, and let $C$ be any commutative subring over which $R$ is a finitely generated (right or left) module. Then $R$ is $C$-projective (Theorem \ref{free}).
  \item  Let $R$ be prime noetherian and a finite module over $Z(R)$. Suppose that $R$ has finite global dimension, is $Z(R)$-Macaulay, and is height 1 Azumaya over $Z(R)$. Then the Azumaya locus of $R$ coincides with the smooth locus of $Z(R)$ (Theorem \ref{Az}).
\end{itemize}
Theorem \ref{free} is a noncommutative generalisation of a familiar characterisation of commutative Cohen-Macaulay rings. It is straightforward to prove when $C$ is central, but seems to be trickier for non-central $C$. The terms used in the statement of Theorem \ref{Az} are explained in $\S$\ref{loci}. It is an improvement of \cite[Theorem 3.8]{BGhom}, the latter having the additional requirement that $R$ is Auslander-regular. Theorem \ref{Az} is illustrated by an application to the Reconstruction Algebras of \cite{W} in $\S$\ref{recon}.

\subsection{Grade symmetry} \label{intro3} To discuss aims (iii) and (iv), recall that the \emph{homological grade} of a non-zero (right) $R$-module $M$ is defined by
$$ j^r_R(M) = \mathrm{min}\{i: \mathrm{Ext}^i_R(M,R_R) \neq 0\} , $$
or $j_R^r(M) = \infty$ if no such $i$ exists. We say that $R$ is \emph{grade symmetric} if $j_R^{\ell}(M) = j_R^{r}(M)$ for all central $R$-bimodules $M$ (where a \emph{central} bimodule satisfies $zm = mz$ for all $z \in Z(R)$ and $m \in M$). Let $\mathrm{Kdim}M$ denote the Krull dimension of the $R$-module $M$. The crux of $\S\S$4 and 5 is the proposal that the hypothesis that $R$ is \emph{grade symmetric $Z(R)$-Macaulay} is what is needed to yield a noncommutative version of the standard commutative chain of homological implications. First, in $\S$4, the key features of the grade symmetric $Z(R)$-Macaulay condition are described, including the interactions of the homological grade with the grade defined in terms of $Z(R)$-sequences (Lemmas \ref{grades} and \ref{muller}, and Theorem \ref{gradequiv}), and what is needed to ensure that
$$\delta := \mathrm{Kdim}R-j_R$$
defines a dimension function on the category of finitely generated $R$-modules (Corollary \ref{dim}). Recall moreover that a ring $R$ is \emph{Krull Macaulay} if $\mathrm{Kdim}R < \infty$ and, for all non-zero finitely generated right or left $R$-modules $M$,
$$ \mathrm{Kdim}R = \mathrm{Kdim}M + j_R(M). $$
 It is shown in Theorem \ref{opus} that, if $R$ is noetherian and a finitely generated module over $Z(R)$, then
$$ R \textit{ Krull-Macaulay} \Leftrightarrow R \; Z(R)-\textit{Macaulay, equicodimensional and grade symmetric}. $$

\subsection{Homological hierarchy} \label{intro4} The definitions of \emph{homological} and \emph{injective homogeneity} (for a ring $R$ which is a finite module over its centre) are recalled from \cite{BrH}, \cite{BrH2} in Definition \ref{injhom}; in essence, these conditions require that simple $R$-modules with the same central annihilator share certain properties in common. The relation of the hom.hom. property with the concept of a \emph{noncommutative crepant resolution} is recalled in Remarks \ref{injhomrems}(iii). In slightly abbreviated form, the main result of $\S$5 states:

\begin{theorem}\label{homo} Let $R$ be a ring which is a finite module over its noetherian centre $Z(R)$. Then the following are equivalent:
\begin{enumerate}
  \item[(i)] $R$ is injectively homogeneous [resp. homologically homogeneous];
  \item[(ii)] \begin{itemize}
  \item $R$ is $Z(R)$-Macaulay;
  \item $R$ is grade symmetric; and
  \item $R$ has finite injective dimension [resp. finite global dimension].
  \end{itemize}
\end{enumerate}
\end{theorem}
Since grade symmetry is trivially satisfied by every commutative ring, it's clear that Theorem \ref{homo} recovers the hierarchy of $\S$\ref{intro1}(iii).

\subsection{Speculation beyond the PI case} \label{guess} Some suggestions about the best way to extend the definition of the Cohen-Macaulay property beyond the setting of rings which are module finite over their centres can be found in $\S$6.

\subsection{Notation} \label{notation} We shall frequently be working in a ring $R$ with a central subring $C$ over which $R$ is a finite module, the centre of $R$ denoted always by $Z(R)$. Given in this setting a prime ideal of $R$, denoted by a Roman capital, say $P$, we'll use the corresponding small fraktur letter, in this instance $\mathfrak{p}$, to denote the prime ideal $P \cap C$ of $C$. Given a nonzero $R$-module $M$, the \emph{support} of $M$ is
$$ \mathrm{supp}(M) := \{\mathfrak{p} \in \mathrm{Spec}(Z(R)): M_{\mathfrak{p}} \neq 0 \}. $$
Thus, if $M$ is a finitely generated $R$-module, then $\mathrm{supp}(M) = \{\mathfrak{p} \in \mathrm{Spec}(Z(R)): \mathfrak{p} \supseteq \mathrm{Ann}_{Z(R)}(M).$

The \emph{codimension} or \emph{height} of a prime ideal $P$ of a ring $R$, denoted $\mathrm{ht}_R(P)$, or simply $\mathrm{ht}(P)$, is the greatest length of a chain of prime ideals descending from $P$; following \cite{E}, our main reference for commutative algebra, we will usually use the former terminology.

The symbols $\mathrm{Kdim}$, $\mathrm{gl.dim}$, $\mathrm{inj.dim}$, $\mathrm{pr.dim}$ are used to denote, respectively, the Krull, global, injective and projective dimensions of a module or of a ring. Details of these and other basic ring-theoretic notions can be found in \cite{MR}, \cite{GW} and \cite{R}, for example.

\subsection*{Acknowledgments}
Early versions of some of the results in this paper formed part of
the second-named author's 2010 PhD thesis at the University of Glasgow,
funded by the Carnegie Trust for the Universities of Scotland. The first author thanks Michael Wemyss and James Zhang for helpful conversations.

\section{Centrally Macaulay rings}
\label{CenMac}

\subsection{Definition and examples} \label{CMdefn}

\begin{definition} \label{centMac} Let $C$ be a commutative noetherian ring, $\mathfrak{i}$ an ideal of $C$ and $X$ a $C$-module.

{\rm(i)} The $\mathfrak{i}$-\textit{grade} of $X$, denoted $G_C (\mathfrak{i},X)$, is the length of a maximal
$C$-sequence in $\mathfrak{i}$ on $X$. When the ring $C$ in question is clear, we write simply $G(\mathfrak{i},X)$.

{\rm(ii)} $X$ is a \textit{Cohen-Macaulay} $C$-module if $G_C(\mathfrak{m},X) = \mathrm{Kdim}_{C_{\mathfrak{m}}}(X_{\mathfrak{m}})$ for all maximal ideals $\mathfrak{m}$ of $C$.

{\rm(iii)} A Cohen-Macaulay module $X$ is \emph{maximal} if it is Cohen-Macaulay and $G(\mathfrak{m},X) =
\mathrm{Kdim}(C_{\mathfrak{m}})$ for all maximal ideals $\mathfrak{m}$ of $C$.

{\rm(iv)} \cite{BHM} A noetherian ring $R$ which is a finite module over a central subring $C$
is $C$-\textit{Macaulay}, or \textit{centrally Macaulay with respect to} $C$, if $R$ is a maximal Cohen-Macaulay
$C$-module.
\end{definition}

\begin{remarks}\label{CMrems} {\rm( i)} The last part of the above definition makes sense because $R$ is noetherian if and only if $C$ is, by \cite{FJ}, \cite[Cor 10.1.11(ii)]{MR}.

{\rm(ii)} Definition \ref{CMdefn}(iv) is equivalent to the condition that $G_C (\mathfrak{m}, R) = \mathrm{Kdim}(R_{\mathfrak{m}})$ for all maximal ideals $\mathfrak{m}$ of $C$, since $\mathrm{Kdim}_{C_{\mathfrak{m}}}(R_{\mathfrak{m}}) = \mathrm{Kdim}_{R_{\mathfrak{m}}}( R_{\mathfrak{m}}).$
\end{remarks}

\begin{examples} \label{CMex} {\rm(i)} Every commutative Cohen-Macaulay ring $R$ is $R$-Macaulay.

{\rm(ii)} Every finite dimensional algebra $R$ over the field $k$ is $k$-Macaulay: in Definition \ref{centMac}(iii), $\mathfrak{m} = 0$ and both integers under consideration are 0.

{\rm(iii)} Every ring $R$ which is a finitely generated torsion-free module over a central integral domain of Krull dimension 1 is $Z(R)$-Macaulay: both integers in Definition \ref{centMac}(iii) are equal to 1.

{\rm(iv)} If $R$ is a free module of finite rank over a central polynomial subalgebra, then $R$ is $Z(R)$-Macaulay; see Theorem \ref{free}(i) for details. Hence all of the following algebras $R$ are $Z(R)$-Macaulay:  $R$ a quantised enveloping algebra or a quantised function algebra with parameter a (non-trivial) root of unity; $R$ the enveloping algebra of a finite dimensional Lie algebra over a field of positive characteristic; $R$ a symplectic reflection algebra with parameter $t = 0$. All these cases follow from properties of $Z(R)$ along with the appropriate PBW-type theorem; details can be extracted from \cite{BGhom}, \cite[Part III]{BG}, \cite{EG}.

{\rm(v)} Let $R$ be an affine Hopf $k$-algebra which is a finitely generated module over its centre. Then $R$ is $Z(R)$-Macaulay. This follows from \cite[Theorem 0.2]{WZ}, coupled with Theorem \ref{opus} below.

{\rm(vi)} \cite[Example 7.3]{BHM} The centre $Z(R)$ of a scalar local domain $R$ which is $Z(R)$-Macaulay need not itself be Cohen-Macaulay. For example, let $k$ be the field of two elements, Let $S = k[X,Y,Z,T]$ and let $\sigma$ be the $k$-algebra automorphism of $S$ with $\sigma(X)=X+Y,\sigma(Y)=Y+Z,\sigma(Z)=Z+T,\sigma(T)= T$; this is chosen to have the property that the fixed algebra $S^{\sigma}$ is \emph{not} Cohen-Macaulay, \cite{Be}, \cite[Ex.16.8]{F}. Let $U$ be the skew polynomial algebra $S[W;\sigma]$, so that $Z(U) = S^{\sigma}[W^4].$ It is now not difficult to show that the maximal ideal $M := \langle X,Y,Z,T,W \rangle$ of $U$ is localisable. Set $R  := U_{M}$. Since $U$ and hence $R$ are homologically homogeneous (see Definition \ref{injhom}(iii)), $R$ is $Z(R)$-Macaulay (by Theorem \ref{hom}). But $Z(R) = (S^{\sigma}[W^4])_{M \cap Z(U)}$ is not Cohen-Macaulay.
\end{examples}

\subsection{Prime ideals}\label{prime} We recall the fundamental relations between the prime ideals of a ring and a central subring, which we'll use frequently in the sequel. For proofs, see for example \cite[1.2, 1.3, 1.7]{Bl}.

\begin{proposition} \label{INCetc} Let $R$ be a ring with a central subring $C$ over which it is a finitely generated module. Let $t$ be the minimal number of generators of $R$ as a $C$-module.

{\rm(i)} \emph{Lying over.} Let $\mathfrak{p}$ be a prime ideal of $C$. Then there exists an integer $s$ with $1 \leq s \leq t$ such that there are exactly $s$ prime ideals $P_1, \ldots , P_s$ of $R$ with $P_i \cap C = \mathfrak{p}$ for each $i$.

{\rm(ii)} \emph{Incomparability.} Let $P$ and $I$ be ideals of $R$ with $P \subsetneq I$ and $P$ prime. Then $P\cap C \subsetneq I \cap C.$

{\rm(iii)} \emph{Going up.} Let $P$ be a prime ideal of $R$, let $\mathfrak{p} := P \cap C$, and let $\mathfrak{q}$ be a prime ideal of $C$ with $\mathfrak{p} \subsetneq \mathfrak{q}.$ Then there exists a prime ideal $Q$ of $R$ with $P \subset Q$ and $Q \cap C = \mathfrak{q}.$
\end{proposition}

As an immediate consequence of (ii) and (iii) of the proposition, we get (i) and (ii) of the following result. Part (iii) is a standard easy exercise.

\begin{corollary}\label{heights} With $R$ and $C$ as in the proposition, let $P$ be a prime ideal of $R$, and let $\mathfrak{p} := P\cap C.$

{\rm(i)} $\mathrm{ht}(\mathfrak{p}) \geq \mathrm{ht}(P).$

{\rm(ii)} There exists a prime ideal $Q$ of $R$, with $Q \cap C = \mathfrak{p}$, such that
$$ \mathrm{ht}(\mathfrak{p}) \leq \mathrm{ht}(Q).$$

{\rm(iii)} Let $\mathcal{C}$ be a multiplicatively closed subset of $C \setminus  \mathfrak{p}$, with $1 \in \mathcal{C}$. Then there is a $1-1$ correspondence between saturated chains of primes descending from $P$ in $R$, and saturated chains of primes descending from $P_{\mathcal{C}}$ in $R_{\mathcal{C}}$. In particular, $\mathrm{ht}_{R}(P) = \mathrm{ht}_{R_{\mathcal{C}}}(P_{\mathcal{C}}).$
\end{corollary}

\begin{example} \label{baddy} If no further hypotheses are imposed, (i) and (ii) of the corollary constitute the best one can say. Thus, let $R = k[X] \oplus k$, where $k$ is a field, and let $C$ be the subalgebra $k[(X,1)] = \{(f,f(1)) : f \in k[X] \}$ of $R$. Take $P = (k[X],0),$ a minimal prime of $R$. Then $P \cap C = ((X-1)k[X],0)$ a maximal ideal of the polynomial subalgebra $C$.
\end{example}

\subsection{Basic properties of centrally Macaulay rings} \label{basic} The following is essentially assembled from \cite{BHM} and \cite{BrH}. Note that (ii) repairs the difficulty highlighted in Example \ref{baddy}.

\begin{theorem} \label{CMprops} Let $R$ be a noetherian ring which is a finite module over the central subring $C$.

{\rm(i)} For all primes $P$ of $R$,
 \begin{equation} \label{firstoff} G_C(\mathfrak{p}, R) \leq  G_{C_{\mathfrak{p}}}(\mathfrak{p}_{\mathfrak{p}}, R_{\mathfrak{p}}) \leq \mathrm{ht}(P_{\mathfrak{p}}) \leq \mathrm{ht}(P).\end{equation}

{\rm(ii)} If $R$ is $C$-Macaulay, then for all prime ideals $P$ of $R$, writing $\mathfrak{p} := P \cap C$,
$$ \mathrm{ht}(P) = G_C(\mathfrak{p},R) = \mathrm{ht}(\mathfrak{p}).$$

{\rm(iii)} If $R$ is $C$-Macaulay and $\mathfrak{p}$ is a prime ideal of $C$, then $R_{\mathfrak{p}}$ is $C_{\mathfrak{p}}$-Macaulay.

{\rm(iv)} $R$ is $C$-Macaulay if and only if $R_{\mathfrak{m}}$ is $C_{\mathfrak{m}}$-Macaulay for all maximal ideals $\mathfrak{m}$ of $C$.

{\rm(v)} $R$ is $C$-Macaulay if and only if $G_C(\mathfrak{m},R) = \mathrm{ht}(M)$ for all maximal ideals $M$ of $R$, where $\mathfrak{m} := M \cap C$.

{\rm(vi)} If $R$ is $C$-Macaulay, then $R$ has an artinian quotient ring.

{\rm(vii)} Suppose that $R$ is $C$-Macaulay, with $\{x_1, \ldots, x_t\}$ a $C$-sequence on $R$. Write $\overline{R} := R/\sum_i x_iR$ and $\overline{C} := C + \sum_i x_iR/\sum_i x_iR$. Then $\overline{R}$ is $\overline{C}$-Macaulay.

\end{theorem}
\begin{proof} (i) It is routine to check \cite[Lemma 18.1]{E} that $C$-sequences are preserved under localisation, so that
\begin{equation} \label{local1} G_C(\mathfrak{p}, R)\leq G_{C_{\mathfrak{p}}}(\mathfrak{p}_{\mathfrak{p}}, R_{\mathfrak{p}}). \end{equation}
An easy argument (independent of the Macaulay property) shows \cite[page 75]{BHM} that
\begin{equation} \label{gulp} G_{C_{\mathfrak{p}}}(\mathfrak{p}_{\mathfrak{p}}, R_{\mathfrak{p}}) \leq \mathrm{ht}(P_{\mathfrak{p}}). \end{equation}
Now (i) follows from (\ref{local1}), (\ref{gulp}) and Corollary \ref{heights}(iii).

(ii) This is \cite[Theorem 4.12]{BHM}.

(iii) Suppose that $R$ is $C$-Macaulay and let $\mathfrak{p}$ be a prime of $C$. Then (ii) shows that, for any prime $P$ of $R$ with $P \cap C = \mathfrak{p}$, equality holds throughout (\ref{firstoff}). Since $\mathrm{Kdim}(R_{\mathfrak{p}})$ is the common codimension of these primes $P_{\mathfrak{p}}$, (iii) follows in the light of Remark \ref{CMrems}(ii).

(iv) The implication from left to right is a special case of (iii). Conversely, suppose that $R_{\mathfrak{m}}$ is $C_{\mathfrak{m}}$-Macaulay for all maximal ideals $\mathfrak{m}$ of $C$, and let $\mathfrak{m}$ be a maximal ideal of $C$.  By \cite[Lemma 18.1]{E} and our hypothesis,
$$ G_C(\mathfrak{m}, R) = G_{C_{\mathfrak{m}}}(\mathfrak{m}_{\mathfrak{m}}, R_{\mathfrak{m}}) = \mathrm{Kdim}_{C_{\mathfrak{m}}}(R_{\mathfrak{m}}), $$
as required.

(v) Suppose first that $R$ is $C$-Macaulay and let $M$ be a maximal ideal of $R$. Then $G_C(\mathfrak{m},R) = \mathrm{Kdim}(R_{\mathfrak{m}})$, and so $G_C(\mathfrak{m},R) \geq \mathrm{ht}(M_{\mathfrak{m}}) = \mathrm{ht}(M)$. The reverse inequality is (\ref{local1}) and (\ref{gulp}). For the opposite implication, fix a maximal ideal $\mathfrak{m}$ of $C$. That $G_C(\mathfrak{m}, R) = \mathrm{Kdim}_{C_{\mathfrak{m}}}(R_{\mathfrak{m}})$ follows at once from the right hand hypothesis and the fact that $$\mathrm{Kdim}_{C_{\mathfrak{m}}}(R_{\mathfrak{m}}) = \mathrm{Kdim}_{R_{\mathfrak{m}}}(R_{\mathfrak{m}}) = \mathrm{max}\{\mathrm{ht}(M): M\vartriangleleft R, M \textit{ maximal, } M \cap C = \mathfrak{m} \}.$$

(vi) Suppose that $R$ is $C$-Macaulay. By (ii), every annihilator prime of $R$ is minimal. Hence, if $P_1, \ldots , P_t$ are the minimal prime ideals of $R$, then (again using (ii)), $\mathfrak{p}_1, \ldots , \mathfrak{p}_t$ are (not necessarily distinct) minimal primes of $C$, and $\mathcal{C} := C \setminus \cup_{i=1}^t \mathfrak{p}_i$ is a non-empty multiplicatively closed subset of $R$ consisting of central regular elements. Therefore $R_{\mathcal{C}}$ is a partial quotient ring of $R$; but it is a finite module over $C_{\mathcal{C}}$, and all its (finitely many) prime ideals $\{P_{i}R_{\mathcal{C}} : i = 1, \ldots , t \}$ are maximal. Since $R_{\mathcal{C}}$ is noetherian, it follows easily that it is artinian, as claimed.

(vii) Let $R$, $C$ and $\{x_1, \ldots , x_t \}$ be as  stated. Without loss of generality, $t = 1$, so write $x = x_1$, and it will be enough to show that $\overline{R} := R/xR$ is $\overline{C}$-Macaulay, where $\overline{C} := (C + xR)/xR.$ By (iv), we may assume that $C$ is local, with maximal ideal $\mathfrak{m}.$ Let $\{x, y_2, \ldots , y_n \}$ be a maximal $C$-sequence in $R$, and let $M$ be a maximal ideal of $R$ with $M \cap C = \mathfrak{m}$. Thus $\overline{M}$ is a maximal ideal of $\overline{R}$, and $\{\overline{y}_2, \ldots , \overline{y}_n \}$ is a maximal $\overline{C}$-sequence in $\overline{R}$. With the inequality following as in (\ref{gulp}),
\begin{equation} \label{forth} G_{\overline{C}}(\overline{\mathfrak{m}}, \overline{R}) = n-1 \leq \mathrm{ht}(\overline{M}). \end{equation}
From (ii) we see that $M$ is a minimal prime over $xR +\sum_{i=2}^n y_i R$. Hence $\overline{M}$ is minimal over $\sum_{i=2}^n \overline{y}_i \overline{R}$. By the Generalised Principal Ideal Theorem, \cite[Theorem 4.1.13]{MR},
\begin{equation} \label{back} \mathrm{ht}(\overline{M}) \leq n - 1. \end{equation}
By (\ref{forth}) and (\ref{back}), $G_{\overline{C}}(\overline{\mathfrak{m}}, \overline{R}) = \mathrm{ht}(\overline{M})$. Hence the result follows from (v).
\end{proof}

\subsection{Equicodimensionality and Krull homogeneity.} \label{equidim}  Let the noetherian ring $R$ be a finite module over its centre. Recall that we say that $R$ is \emph{equicodimensional} if every maximal ideal $M$ of $R$ has the same codimension; this common codimension is then necessarily $\mathrm{Kdim}(R)$. We'll see that equicodimensionality is key to the validity of certain desirable properties of a centrally Macaulay ring.

\begin{lemma} \label{equi} Suppose that the noetherian ring $R$ is a finite module over a central subring $C$.

{\rm(i)} If $R$ is equicodimensional, then $C$ is equicodimensional. In this case every maximal ideal of $R$ and every maximal ideal of $C$ has codimension equal to $\mathrm{Kdim}(R)$.

{\rm(ii)} The converse to the first sentence of (i) is true if $R$ is $C$-Macaulay.
\end{lemma}

\begin{proof} (i) Suppose that $R$ is equicodimensional, and let $\mathfrak{m}$ be a maximal ideal of $C$. By Corollary \ref{heights}(ii) there is a maximal ideal $M$ of $R$ with $M \cap C = \mathfrak{m}$ such that $\mathrm{ht}(M) \geq \mathrm{ht}(\mathfrak{m})$. But, by Corollary \ref{heights}(i), $\mathrm{ht}(M) \leq \mathrm{ht}(\mathfrak{m})$ also holds. Hence $\mathrm{ht}(\mathfrak{m}) = \mathrm{Kdim}(R)$ by our hypothesis on $R$, and so $C$ is equicodimensional.

(ii) Suppose that $R$ is $C$-Macaulay and $C$ is equicodimensional. Then $R$ is equicodimensional by Theorem \ref{CMprops}(ii).
\end{proof}

The example in Remark \ref{blah}(ii) below shows that Lemma \ref{equi}(ii) fails in the absence of the $C$-Macaulay hypothesis.

A noetherian ring $R$ is \emph{Krull homogeneous} if $\mathrm{Kdim}(I) = \mathrm{Kdim}(R)$ for every non-zero right or left ideal $I$ of $R$. Note that, if $R$ is a finite module over its centre, then there is no need to deal separately with right and left ideals when handling this concept: ideals of $R$ have the same right and left Krull dimensions by \cite[Corollary 6.4.13]{MR}, and if $R$ has a non-zero right ideal $I$ with $\mathrm{Kdim}(I) < \mathrm{Kdim}(R)$, then it is easy to see that $\mathrm{Kdim}(RI) < \mathrm{Kdim}(R)$.

\begin{proposition}\label{affequi} Let $R$ be a finite module over $Z(R)$, and suppose that $R$ is an affine $k$-algebra, where $k$ is a field. If $R$ is Krull homogeneous, then it is equicodimensional.
\end{proposition}
\begin{proof} Let $R$ be a Krull homogeneous $k$-algebra which is finite over its centre, and let $Q$ be a minimal prime of $R$. Then $\mathrm{Kdim}(R) = \mathrm{Kdim}(R/Q)$; this is true for \emph{any} Krull homogeneous ring, but is easy to see directly in the present setting - with $\mathfrak{q} := Q \cap Z(R)$, $\mathrm{Kdim}(R/Q) = \mathrm{Kdim}(Z(R)/\mathfrak{q})$,  and $\mathfrak{q}$ has non-zero annihilator in $R$. Moreover,
$$ \mathrm{Kdim}(R/Q) = \mathrm{ht}(M/Q) $$
for every maximal ideal $M$ of $R$ with $Q \subseteq M$, by the catenarity of affine PI rings, \cite[Theorem 4]{Sch}. Thus every maximal ideal of $R$ has codimension equal to $\mathrm{Kdim}R$, as required.
\end{proof}

\begin{remarks} {\rm(i)} The converse to Proposition \ref{affequi} is false: consider the commutative algebra $R =  k[X,Y]/\langle X^2, XY \rangle$. It is equicodimensional, with $\mathrm{Kdim}R = 1$, but $\overline{X}R$ is a non-zero artinian ideal of $R$.

{\rm(ii)} In a partial converse to Proposition \ref{affequi}, it will be shown in Corollary \ref{partcon} that equicodimensional $Z(R)$-Macaulay rings are Krull homogeneous.

{\rm(iii)} The above proposition is trivially false if the affine hypothesis is dropped, even if $R$ is commutative Cohen Macaulay - consider for example the localisation of $k[X,Y]$ at $<(X+1)X, (X+1)Y>$.
\end{remarks}

\begin{proposition}  \label{equidimprop} Suppose that the noetherian ring $R$ is a finite module over a central subring $C$, and that $R$ is $C$-Macaulay and equicodimensional of dimension $n$. Let $\{x_1, \ldots , x_t \}$ be a $C$-sequence on $R$. Set $\overline{R} := R/\sum_i x_iR$ and $\overline{C} := C + \sum_i x_iR/\sum_i x_iR$. Then $\overline{R}$ is $\overline{C}$-Macaulay and is equicodimensional of dimension $n-t$.
\end{proposition}

\begin{proof} The first claim is Theorem \ref{CMprops}(vii). For the second part, let $M$ be any maximal ideal of $R$ containing $\{x_1, \ldots , x_t \}$, so $G_C(\mathfrak{m},R) = n$ by our hypotheses. As in (\ref{forth}), $G_C(\overline{\mathfrak{m}}, \overline{R}) = n-t$. By the first part of the proposition, $\mathrm{ht}(M) = n-t$, so the second claim follows.
\end{proof}

\subsection{Dependence on the central subring} \label{depend}

It is immediate from Theorem \ref{CMprops}(v) and (i) that if a ring $R$ is $C$-Macaulay for some central subring $C$ over which it is a finite module, then $R$ is $Z(R)$-Macaulay. But the reverse implication is somewhat more subtle:

\begin{theorem} \label{hinge} Let $R$ be a noetherian ring which is a finite module over the central subring $C$. Suppose that $R$ is $Z(R)$-Macaulay. Then $R$ is $C$-Macaulay if and only if, for all maximal ideals $M$ and $N$ of $Z(R)$ with $M \cap C = N \cap C$, $\mathrm{ht}(M) = \mathrm{ht}(N)$.
\end{theorem}
\begin{proof} Suppose first that $R$ is $C$-Macaulay. Let $M$ and $N$ be maximal ideals of $Z(R)$ with $M \cap C = N \cap C := \mathfrak{m}.$ Let $\widehat{M}$ and $\widehat{N}$ be maximal ideals of $R$ lying over $M$ and $N$ respectively. So
$$ \widehat{M} \cap C = M \cap C = \mathfrak{m} = N \cap C = \widehat{N} \cap C. $$
Since $R$ is $C$-Macaulay and $Z(R)$-Macaulay, two applications of Theorem \ref{CMprops}(ii) yield
$$ \mathrm{ht}(M) = \mathrm{ht}(\widehat{M}) = \mathrm{ht}(\mathfrak{m}) = \mathrm{ht}(\widehat{N})=  \mathrm{ht}(N),$$
as claimed.

Conversely, assume that $R$ is $Z(R)$-Macaulay, and that the equality of codimensions property holds for the inclusion $C \subseteq Z(R)$. Let $M$ be a maximal ideal of $R$, and set $$\mathfrak{n} := M \cap C, \quad \quad \mathfrak{m} := M \cap Z(R).$$ Localise at $\mathfrak{n}$ in $C$ to ensure that, without loss of generality, $C$ is local with $J(C) = \mathfrak{n}:$ this is legitimate by Theorem \ref{CMprops}(iv).

By Theorem \ref{CMprops}(ii), since $R$ is $Z(R)$-Macaulay,
\begin{equation} \label{first} t := \mathrm{ht}(M) = \mathrm{ht}(\mathfrak{m}) = G_{Z(R)}(\mathfrak{m},R). \end{equation}
Let $\{x_1, \ldots , x_s \}$ be a maximal $C$-sequence in $\mathfrak{n}$ on $R$, so that $s \leq t$ since $\mathfrak{n} \subseteq \mathfrak{m}.$ Set $I := \sum_{i=1}^s x_iR$. Then $(\mathfrak{n} + I)/I$ consists of zero divisors in $R/I$, so that, by \cite[Proposition 3.4]{BHM}, there exists $c \in R \setminus I$ with $c\mathfrak{n} \subseteq I.$ Since $R/\mathfrak{n}R$ is Artinian, $(cR + I)/I$ is a non-zero Artinian submodule of $R/I$. Hence there exists a simple right $R$-module $T/I$ and a maximal ideal $M'$ of $R$ with $TM' \subseteq I$. Set $\mathfrak{m}':= M' \cap Z(R)$, so $\mathfrak{m}' \cap C = \mathfrak{n}$. Since $\mathfrak{m}'$ consists of zero divisors \emph{mod}$I$, $\{x_1, \ldots , x_s \}$ is a maximal $Z(R)$-sequence in $M'$. But $R$ is $Z(R)$-Macaulay, so
\begin{equation}\label{duo} s = G_{Z(R)}(\mathfrak{m}', R) = \mathrm{ht}(M') = \mathrm{ht}(\mathfrak{m}').\end{equation}
However, by hypothesis, since $\mathfrak{m}' \cap C = \mathfrak{n} = \mathfrak{m} \cap C$,
\begin{equation} \label{trio} \mathrm{ht}(\mathfrak{m}') = \mathrm{ht}(\mathfrak{m}). \end{equation}
Now (\ref{first}), (\ref{duo}) and (\ref{trio}) imply that $s=t.$ That is, $G_C(M \cap C, R) = \mathrm{ht}(M)$. Since $M$ was an arbitrary maximal ideal of $R$, the result follows from Theorem \ref{CMprops}(v).
\end{proof}

It follows at once from Theorem \ref{hinge} that the Macaulay property of $R$ is independent of the choice of central subring, provided that $Z(R)$ is equicodimensional:

\begin {corollary} \label{independ} Let $R$ be a noetherian ring which is a finite module over the central subring $C$. If $Z(R)$ is equicodimensional, then $R$ is $Z(R)$-Macaulay if and only if $R$ is $C$-Macaulay.
\end{corollary}

\begin{remarks}\label{blah}{\rm(i)}  The corollary shows that independence of the central subalgebra for the Macaulay property holds in particular if $R$ is a prime affine $k$-algebra, $k$ a field, by the Artin-Tate lemma \cite[Lemma 13.9.10]{MR}.

{\rm(ii)} Corollary \ref{independ} is false without the equicodimensionality hypothesis. Consider again Example \ref{baddy}, $R = k[X] \oplus k$, where $k$ is a field, with $C = k[(X,1)].$ This is $R$-Macaulay, but is not $C$-Macaulay. For the maximal ideal $P = \langle X-1 \rangle \oplus k$ of $R$ has $\mathrm{ht}(P)=1$ but $G_C(P,R) = 0$. Errors of this type are very easy to make, and have occurred quite frequently in the literature. Examples which are the fault of the first author of the current paper are described in remarks (iii) and (iv) below. A further example is \cite[2.1 - 2.5]{BGG}, where several results are claimed of the ``independence of coefficient ring'' sort, which are manifestly false, with counterexamples similar to the one just given. However, in the key Lie-theoretic applications  in \cite{BGG}, equicodimensionality is valid and the applications survive.

{\rm(iii)} Remark (ii) shows that \cite[Proposition 2.7]{BrH} is not correct. This result concerns the \emph{homologically homogeneous} property, whose definition is recalled in Definition \ref{injhom}(iii) below. \cite[Proposition 2.7]{BrH} asserts that the homologically homogeneous property for a ring is independent of our choice of central subring, but it is easy to see that the hereditary commutative algebra $R = k[X] \oplus k$, which is - trivially - hom.hom. over $R$, is \emph{not} homologically homogeneous with respect to $C = k[(X,1)]$. For, let $M := \langle X-1 \rangle \oplus k$ and $N :=k[X] \oplus \{0\}$, maximal ideals of $R$ with $M \cap C = N \cap C = \langle X-1 \rangle \oplus \{0\}$; but $\mathrm{pr.dim}_R(R/M) = 1$, whereas $\mathrm{pr.dim}_R(R/N) = 0.$ The error in the proof of \cite[Proposition 2.7]{BrH} lies in the claim that, for a hom hom. ring $R$, $\mathrm{ht}(M) = \mathrm{ht}(M \cap C)$ for every maximal ideal $M$ of $R$ or $Z(R)$; the equicodimensionality property is enough to ensure this, since the Krull dimensions of $R$, $Z(R)$ and $C$ are necessarily equal, and hence so are codimensions in the equicodimensional setting. To recover \cite[Proposition 2.7]{BrH}, it is thus necessary to impose the additional hypothesis that $R$ is equicodimensional, taking note of Lemma \ref{equi}(i).

{\rm(iv)} A parallel correction is required to \cite[Corollary 3.6]{BrH2}, concerning the independence of the \emph{injectively homogeneous} property on the choice of central subring. The relevant definition is recalled in Definition \ref{injhom}(ii) below, and the same example as in (iii) shows that \cite[Corollary 3.6]{BrH2} is false. Again, the result and its proof remain valid if $R$ is assumed to be equicodimensional.
\end{remarks}

\section{Consequences of the centrally Macaulay property} \label{CenMacCon}

\subsection{Chains of primes; catenarity} \label{cat}
Recall that a ring $R$ is \emph{catenary} if, given any two primes $P$ and $Q$ of $R$ with $Q \subset P$, all saturated chains of primes between $Q$ and $P$ have the same length.

The key parts (i) and (iii) of the following result are due to Goto and Nishida, \cite[Corollary 1.3]{GN}.

\begin{theorem} \label{chains} \label{catenary} Let $R$ be a noetherian ring which is a finite module over a central subring $C$. Suppose that $R$ is $C$-Macaulay.

{\rm (i)} $R$ is catenary. More precisely, every saturated chain of prime ideals between primes $P$ and $Q$ of $R$ with $Q \subseteq P$ has length $\mathrm{ht}(P) - \mathrm{ht}(Q)$.

{\rm(ii)} $C$ is catenary. More precisely, every saturated chain of prime ideals between primes $\mathfrak{p}$ and $\mathfrak{q}$ of $C$ with $\mathfrak{q} \subseteq \mathfrak{p}$ has length $\mathrm{ht}(\mathfrak{p}) - \mathrm{ht}(\mathfrak{q})$.

{\rm(iii)} If $R$ (or equivalently $C$) is equicodimensional, then, for every prime ideal $P$ of $R$,
$$ \mathrm{Kdim}(R) = \mathrm{Kdim}(R/P) + \mathrm{ht}(P).$$

{\rm(iv)} If $R$ (or equivalently $C$) is equicodimensional), then, for every prime ideal $\mathfrak{p}$ of $C$,
$$ \mathrm{Kdim}(C) = \mathrm{Kdim}(C/\mathfrak{p}) + \mathrm{ht}(\mathfrak{p}).$$
\end{theorem}

\begin{proof}{\rm (i)} Since we can without loss localise to $R_{\mathfrak{p}}$ by Theorem \ref{CMprops}(iii) and Corollary \ref{heights}(iii), this is the first and last parts of \cite[Corollary 1.3]{GN}.

{\rm(iii)} Suppose that $R$ is equicodimensional. The equivalence of this hypothesis with the equicodimensionality of $C$ is Lemma \ref{equi}. Let $P$ be a prime ideal of $R$, and let $M$ be a maximal ideal of $R$ with $P \subseteq M$. Write $\mathfrak{m} = M \cap R$. Since $R$ is equicodimensional, and by Corollary \ref{heights}(iii),
\begin{equation}\label{go} \mathrm{Kdim}(R) = \mathrm{ht}(M) = \mathrm{ht}(M_{\mathfrak{m}}) = \mathrm{Kdim}(R_{\mathfrak{m}}). \end{equation}
Now $R/P$ is equicodimensional, by (i) and the equicodimensionality of $R$. Hence
\begin{equation}\label{going} \mathrm{Kdim}(R/P) = \mathrm{ht}(M/P) = \mathrm{ht}(M_{\mathfrak{m}}/P_{\mathfrak{m}}) = \mathrm{Kdim}(R_{\mathfrak{m}}/P_{\mathfrak{m}}). \end{equation}
Since the result we want to prove is true when $C$ is local, by \cite[Corollary 1.3]{GN}, Theorem \ref{CMprops}(iii) implies that
\begin{equation}\label{gone} \mathrm{Kdim}(R_{\mathfrak{m}}) = \mathrm{Kdim}(R_{\mathfrak{m}}/P_{\mathfrak{m}}) + \mathrm{ht}_{R_{\mathfrak{m}}}(P_{\mathfrak{m}}).
\end{equation}
It follows from (\ref{go}), (\ref{going}), (\ref{gone}) and another use of Corollary \ref{heights}(iii) that
$$ \mathrm{Kdim}(R) = \mathrm{Kdim}(R/P) + \mathrm{ht}(P). $$

(iv) Let $\mathfrak{p}$ be a prime ideal of $C$. Then (iv) follows from (iii) by taking a prime $P$ of $R$ which lies over $\mathfrak{p}$ and applying (iii) to $P$, noting that $R$ and $C$ (and similarly $R/P$ and $C/\mathfrak{p}$) have the same Krull dimensions, and using also Theorem \ref{CMprops}(ii).

(ii) Let $\mathfrak{p}$ and $\mathfrak{q}$ be as stated. By Theorem \ref{CMprops}(iii) and Corollary \ref{heights}(iii) we can localise $R$ at $\mathfrak{p}$, so without loss $C$ is local with maximal ideal $\mathfrak{p}$, and hence $C$ and $R$ are equicodimensional, by Lemma \ref{equi}. By equicodimensionality and (iv),
\begin{equation}\label{hop}  \mathrm{ht}(\mathfrak{p}) = \mathrm{Kdim}(C) = \mathrm{Kdim}(C/\mathfrak{q}) + \mathrm{ht}(\mathfrak{q}), \end{equation}
so that
\begin{equation} \label{henry} \mathrm{Kdim}(C/\mathfrak{q}) = \mathrm{ht}(\mathfrak{p}) - \mathrm{ht}(\mathfrak{q}).
\end{equation}
 Suppose that there is a saturated chain of primes $\mathfrak{q} = \mathfrak{q}_0 \subsetneq \cdots \subsetneq \mathfrak{q}_t = \mathfrak{p}$, so that
\begin{equation} \label{tough} t \leq \mathrm{ht}(\mathfrak{p}) - \mathrm{ht}(\mathfrak{q})
\end{equation}
by (\ref{henry}). Suppose that the inequality (\ref{tough}) is strict. Then, bearing in mind that $\mathfrak{p}$ is maximal, there is an inclusion $\mathfrak{q}_i \subsetneq \mathfrak{q}_{i+1}$ where $\mathrm{Kdim}(C/\mathfrak{q}_{i}) \geq \mathrm{Kdim}(C/\mathfrak{q}_{i+1}) + 2.$ Replacing $\mathfrak{q}$ and $\mathfrak{p}$ by  $\mathfrak{q}_{i}$ and  $\mathfrak{q}_{i+1}$ respectively, and noting (iv), we can assume that
\begin{equation}\label{yuck} \mathrm{ht}(\mathfrak{p}) - \mathrm{ht}(\mathfrak{q}) \geq 2 > 1 = \mathrm{ht}(\mathfrak{p}/\mathfrak{q}),\end{equation}
(though $\mathfrak{p}$ may no longer be maximal).

By Going Up, Corollary \ref{heights}(iii), there are primes $Q$ and $P$ of $R$ lying over $\mathfrak{q}$ and $\mathfrak{p}$ respectively, with $Q \subset P$. By Theorem \ref{CMprops}(ii) and (\ref{yuck}), $\mathrm{ht}(P) - \mathrm{ht}(Q) \geq 2$. By (i), $\mathrm{ht}(P/Q) \geq 2$; in particular there is a prime ideal $J$ of $R$ with $Q \subsetneq J \subsetneq P$. It follows that $\mathfrak{q} \subsetneq \mathfrak{j} := J \cap C \subsetneq \mathfrak{p}$, contradicting (\ref{yuck}). The assumption that (\ref{tough}) is strict must therefore be false, and so the proof is complete.
\end{proof}

We can now prove the promised partial converse to Proposition \ref{affequi}.

\begin{corollary}\label{partcon} Let $R$ be a noetherian ring which is a finite module over $Z(R)$. Suppose that $R$ is equicodimensional and $Z(R)$-Macaulay. Then $R$ is Krull homogeneous.
\end{corollary}
\begin{proof} Since $R$ is $Z(R)$-Macaulay it has an artinian quotient ring by Theorem \ref{CMprops}(vi). In particular this means that every (right or left) annihilator prime ideal of $R$ is a minimal prime, \cite[Proposition 3.2.4(v) and Corollary 4.1.4]{MR}. Now Theorem \ref{chains}(iii) implies that, for every minimal prime $P$ of $R$,
$$ \mathrm{Kdim}(R/P) = \mathrm{Kdim}(R).$$
If $R$ has a non-zero right ideal $I$ with $\mathrm{Kdim}(I) < \mathrm{Kdim}(R)$, then $R$ has such a right ideal with prime annihilator. The above facts thus show that no such $I$ can exist, proving that $R$ is Krull homogeneous.
\end{proof}

\begin{remarks}\label{catchat} Theorem \ref{chains} of Goto and Nishida replaces an earlier incomplete proof purporting to prove the same result, in \cite{BHM}.
\end{remarks}

\subsection{Indecomposible centrally Macaulay rings: the affine case} \label{indec}
In this subsection we  prove

\begin{theorem}\label{affinesplit} Let the ring $R$ be a finite module over the central subring $C$. Suppose that $C$ (or, equivalently, $R$), is affine over a field, and that $R$ is $C$-Macaulay. Then $R$ and $C$ are direct sums,
$$R = \oplus_i R_i \supseteq \oplus_i C_i = C,$$
where $C_i = C \cap R_i$. For all $i$, the algebras $R_i$ and $C_i$ are equicodimensional and Krull homogeneous; and $R_i$ is $C_i$-Macaulay.
\end{theorem}

The equivalence of the affine hypotheses on $R$ and $C$ is the Artin-Tate lemma, \cite[Lemma 13.9.10]{MR}. The theorem is a generalisation of a well-known result in the commutative theory, \cite[Exercise 18.6]{E}. It is easily seen to be false if the affine hypothesis is omitted: for example, let $k$ be a field and let $R$ be the localisation of $k[X,Y]$ at the semiprime ideal $(X+1)X, (X+1)Y \rangle$.

The proof of Theorem \ref{affinesplit} will need the theory of primary decomposition in noncommutative rings initiated in 1962 by Gabri$\grave{\mathrm{e}}$l \cite{Ga} and developed by Gordon and others \cite{Go}, \cite{BrP}, \cite{Ja}. We recall here the minimum required.

\begin{definition} \label{primary} Let $P$ be a prime ideal of the noetherian ring $R$.

{\rm(i)} The ring $R$ is \emph{right} $P$-\emph{primary} if $P$ is the unique right annihilator prime ideal of (the right $R$-module) $R$; that is, $P$ is the unique \emph{right associated prime} of $R$. If $R$ is $P$-primary for some prime ideal $P$, then $R$ is called \emph{primary}.

{\rm(ii)} A \emph{primary decomposition} of $R$ is a finite intersection $0 = \bigcap_{j=1}^n I_j$, where $\{I_j : 1 \leq j \leq n \}$ is a collection of ideals of $R$, with $R/I_j$ being $P_j/I_j$-primary for $j = 1, \ldots , n$, and with the intersection of every proper subset of the $\{I_j\}$ being non-zero.
 \end{definition}

The following proposition is valid under much weaker hypotheses on $R$, but we limit the discussion to what is needed for the present proof.

\begin{proposition} \label{prim} Let the noetherian ring $R$ be a finitely generated module over the central subring $C$.

{\rm(i)} The ring $R$ admits an irredundant primary decomposition.
Let
\begin{equation}\label{cap} 0 \; = \; \bigcap_{j=1}^n I_j
 \end{equation}
be one primary decomposition as ensured by (i), with $R/I_j$ being $P_j/I_j$-primary, $j = 1, \ldots , n.$

{\rm(ii)} If $R$ is primary, then it has an artinian quotient ring. In particular, $P_j/I_j$ is a minimal prime ideal of $R/I_j$, for $j = 1, \ldots , n$.

{\rm(iii)} If $R$ is $P$-primary, then $\mathrm{Kdim}(R) = \mathrm{Kdim}(R/P) = \mathrm{Kdim}(R/Q)$ for every minimal prime ideal $Q$ of $R$.

{\rm(iv)} Suppose that $R$ has an artinian quotient ring. Then, in the notation of (\ref{cap}),  for every $j = 1, \ldots , n$, all the minimal prime ideals over $I_j$, including $P_j$, are minimal prime ideals of $R$
\end{proposition}

\begin{proof} (i) \cite[Corollary 2.4]{Go}.

(ii), (iii) Suppose that $R$ is $P$-primary. By \cite[Theorem 6.2(i)]{BrP}, $R$ has an artinian quotient ring, and is Krull homogeneous, so that $\mathrm{Kdim}(X) = \mathrm{Kdim}(R)$ for every non-zero right or left ideal of $R$. In particular, $\mathrm{Kdim}(R/P) =  \mathrm{Kdim}(R)$, since by definition $P$  is the right annihilator of a non-zero ideal of $R$. Hence, $P$ is a minimal prime. Krull homogeneity, combined with \cite[Theorem 5.4(ii)]{BrP}, means that every minimal prime ideal $Q$ of $R$ satisfies $\mathrm{Kdim}(R/Q) = \mathrm{Kdim}(R/P)$.

(iv) Suppose that $R$ has an artinian quotient ring $Q(R)$. By (\ref{cap}), the right $R$-module $R$ embeds in $\oplus_j R/I_j$. Denote this embedding by $\iota$, and observe that the irredundancy requirement of the intersection implies that, for each $j = 1, \ldots , n$, $\iota(\cap_{i \neq j}I_i \subseteq \iota (R) \cap (R/I_j) \neq \{0\}$. Hence, for each $j$, $P_j$ is a right associated prime of $R$. In particular, $P_j$ is a minimal prime ideal of $R$, for all $j = 1, \ldots , n$.  But, for every $j$, every minimal prime ideal $Q/I_j$ of $R/I_j$ is in the clique of $P_j/I_j$, since every such prime $Q/I_j$ occurs as the annihilator of a critical composition factor of the right module $R/I_j$, and the latter module embeds in a finite direct sum of copies of the $R/I_j$-injective hull of $R/P_j$, due to $R/I_j$ being right $P_j/I_j$-primary. \emph{A fortiori}, $Q$ is in the clique of $P_j$. Now the existence of $Q(R)$ implies that the set of minimal prime ideals of $R$ is a union of cliques, by \cite[Proposition 7.4.8]{Ja}. Since $P_j$ is a minimal prime of $R$, so is $Q$, as claimed.
\end{proof}

\textbf{Proof of Theorem \ref{affinesplit}:} Let $R$ be a finitely generated module over the central affine subalgebra $C$, and assume that $R$ is $C$-Macaulay. By Proposition \ref{prim}(i), $R$ admits an irredundant primary decomposition $0 = \bigcap_{j=1}^n I_j$. Since $C$ is affine, $\mathrm{Kdim}(C) = \mathrm{Kdim}(R) =: s < \infty$. For each $\ell = 0, \ldots , s$, define
$$ J_{\ell} := R \; \cap \; \bigcap\{ I_j : \mathrm{Kdim}(R/I_j) = \ell \}.$$
Relabelling the ideals $J_{\ell}$ to omit those which are equal to $R$, we obtain an irredundant decomposition
\begin{equation}\label{irred} 0 \; = \; \bigcap_{\ell} J_{\ell}. \end{equation}
By Theorem \ref{CMprops}(vi) $R$ has an artinian quotient ring, since it is $C$-Macaulay. Hence, by Proposition \ref{prim}(iv) every prime ideal minimal over $I_j$ is a minimal prime of $R$. By definition of the ideals $J_{\ell}$, the same is true for them.

We claim that
\begin{equation} \label{decomp} R \; \cong \; \bigoplus_{\ell} R/J_{\ell}.
\end{equation}
To prove (\ref{decomp}), it is enough to show that, for all $\ell$,
\begin{equation} \label{coprime} T_{\ell} \; := \; J_{\ell} + (\cap_{r \neq \ell}J_r) \; = \; R.
\end{equation}
Suppose that (\ref{coprime}) fails for $\ell$ and let $M$ be a maximal ideal of $R$ with $T_{\ell} \subseteq M.$ Thus $J_{\ell} \subseteq M$ and $J_r \subseteq M$ for some $r \neq \ell,$ so there exist prime ideals $P$ and $Q$ of $R$ with $P$ minimal over $J_{\ell}$ and $Q$ minimal over $J_r$, such that
$$ P + Q \subseteq M.$$
Note that, by the discussion between (\ref{irred}) and (\ref{decomp}), $P$ and $Q$ are both minimal primes of $R$; moreover $P$ is minimal over $I_j$ and $Q$ is minimal over $I_t$ for primary ideals $I_j$ [resp. $I_t$] occurring in the intersections defining $J_{\ell}$ [resp. $J_r$]. By construction, there are distinct integers $a$ and $b$ with
\begin{equation}\label{krull1} \mathrm{Kdim}(R/J_{\ell}) = a, \; \; \mathrm{Kdim}(R/J_{r}) = b. \end{equation}
By Proposition \ref{prim}(v), applied in $R/I_j$ and in $R/I_t$,
\begin{equation} \label{krull2} \mathrm{Kdim}(R/P) = a, \; \; \mathrm{Kdim}(R/Q) = b. \end{equation}
Since $R$ is $C$-Macaulay and $P$ is a minimal prime of $R$, $\mathrm{ht}(M) = \mathrm{ht}(M/P)$ by Theorem \ref{chains}; and since $R/P$ is a prime affine PI-ring, $\mathrm{Kdim}(R/P) = \mathrm{ht}(M/P)$ by \cite[Theorem 4]{Sch}. That is,
\begin{equation} \label{krull3} \mathrm{Kdim}(R/P) \; = \; \mathrm{ht}(M/P) \; = \; \mathrm{ht}(M);\end{equation}
similarly,
\begin{equation} \label{krull4} \mathrm{Kdim}(R/Q) \; = \; \mathrm{ht}(M/Q) \; = \; \mathrm{ht}(M).\end{equation}
Given that $a \neq b$,  (\ref{krull2}), (\ref{krull3}) and (\ref{krull4}) yield a contradiction. So (\ref{coprime}) must be true, and the direct sum decomposition (\ref{decomp}) is proved. The primary factors $R/I_j$ are Krull homogeneous by Proposition \ref{prim}(ii) and (iii), and hence so are the summands $R/J_{\ell}$, by construction. Finally, they are then also equicodimensional, in view of Proposition \ref{affequi}.  $\qquad \qquad \qquad \Box$

\subsection{Projectivity over regular subrings}\label{free} Recall the characterisation of  a commutative Cohen-Macaulay ring \cite[Corollary 18.17]{E}: an equicodimensional (commutative) noetherian ring is Cohen-Macaulay if and only if it is a projective module over some [resp. every] regular subring $C$ over which it is finitely generated. It is straightforward to extend this to the noncommutative setting, with $C$ central in $R$, but in fact we can do better than this, as in (ii) of the following result.

\begin{theorem} \label{free} Let $R$ be a noetherian ring which is a finitely generated $Z(R)$-module.

{\rm(i)} If there exists a regular subring $C$ of $Z(R)$ over which $R$ is a finitely generated projective module, then $R$ is $C$-Macaulay, and hence $Z(R)$-Macaulay.

{\rm(ii)} Suppose that $R$ is equicodimensional and $Z(R)$-Macaulay. Let $C$ be any (commutative) regular subring of $R$ with $R$ a finitely generated (right or left) $C$-module. Then $R$ is a projective $C$-module.
\end{theorem}
\begin{proof} {\rm(i)} Suppose that $R$ is a finitely generated projective $C$-module, with $C$ central and regular. These properties are preserved by localisation at a maximal ideal of $C$, so, by Theorem \ref{CMprops}(iv), in proving that $R$ is $C$-Macaulay we may assume that $C$ is local. In this case $C$, being regular local, is Cohen-Macaulay, \cite[Corollary 10.15]{E}. Hence $R$, being a free $C$-module by \cite[Exercise 4.11a]{E}, is a Cohen-Macaulay $C$-module, as required. That $R$ is $Z(R)$-Macaulay now follows by the first paragraph of $\S$\ref{depend}.

{\rm(ii)} \textbf{Step 1: Proof of} (\ref{Amax}): Let $C$ and $R$ be as stated, with $R$ a finitely generated right $C$-module. Set $A$ to be the subring of $R$ generated by $C$ and $Z(R)$, so that $A$ is commutative noetherian and a finitely generated $C$-module, and $R$ is finitely generated as a right and as a left $A$-module. We show first that
\begin{equation}\label{Amax} A \textit{ is equicodimensional, and } R \textit{ is a maximal Cohen-Macaulay } A-\textit{module.} \end{equation}
Let $P$ be a maximal ideal of $A$, and set $\mathfrak{p} := P \cap Z(R)$, a maximal ideal of $Z(R)$ by Going Up, Proposition \ref{INCetc}(iii). The equicodimensionality hypothesis applies to both $R$ and $Z(R)$ by Lemma \ref{equi}. Since $R$ is $Z(R)$-Macaulay, Theorem \ref{CMprops}(ii) and equicodimensionality ensure that
\begin{equation} \label{crewe} G_{Z(R)}(\mathfrak{p}, R) = \mathrm{ht}(\mathfrak{p}) = \mathrm{Kdim}(Z(R)) = \mathrm{Kdim}(R).
\end{equation}
Since any $Z(R)$-sequence on $R$ is an $A$-sequence, and using \cite[Proposition 18.2]{E} for the second inequality,
\begin{equation}\label{stay} G_{Z(R)}(\mathfrak{p},R) \leq G_{A}(P,\,_{R}R) \leq \mathrm{ht}(P) \leq \mathrm{Kdim}_{A}(A).
\end{equation}
Now $R$ is an $R-A$-bimodule, finitely generated on each side, so
\begin{equation} \label{Krull} \mathrm{Kdim}_{A}(A) = \mathrm{Kdim}_{R}(R)
 \end{equation}
 by \cite[Corollary 6.4.13]{MR}. It therefore follows that equality holds throughout (\ref{crewe}) and (\ref{stay}). In particular,
\begin{equation}\label{triad} G_{A}(P, \,_{R}R) = \mathrm{ht}(P) = \mathrm{Kdim}(R).\end{equation}
Since $P$ was an arbitrary maximal ideal of $A$, (\ref{Amax}) is proved.

\textbf{Step 2: $R$ is a maximal Cohen-Macaulay (left) $C$-module:} Now let $\mathfrak{m}$ be a maximal ideal of $C$. Using Proposition \ref{INCetc}(i), let $P$ be a maximal ideal of $A$ with $P$ lying over $\mathfrak{m}$. By (\ref{Amax}) and Corollary \ref{heights}(i),(ii),
\begin{equation}\label{Astart} \textrm{Kdim}(A) = \mathrm{ht}(P) = \mathrm{ht}(\mathfrak{m}). \end{equation}
From (\ref{Astart}), (\ref{Krull}) and (\ref{triad}),
\begin{equation}\label{keyer}  s := G_C(\mathfrak{m},\,_{R}R) \leq G_A(P,\,_{R}R) = \textrm{Kdim}(R) = \textrm{Kdim}(A) = \mathrm{ht}(P) = \mathrm{ht}(\mathfrak{m}).
\end{equation}
Suppose for a contradiction that the inequality in (\ref{keyer}) is strict, and let $\{x_1,\ldots , x_s \}$ be a maximal $C$-sequence  on $\,_{R}R$ in $\mathfrak{m}$. Note that $\,_{C}R$ is also finitely generated, since $\,{C}A$ and $\,_{A}R$ are both finitely generated modules. Let $I := \sum_{i=1}^s  x_i R$, so that $R/I$ is a finitely generated left $C/I\cap C$-module by the previous sentence. Then $\mathfrak{m}/I \cap C$ consists of zero divisors on $R/I$, so by \cite[Proposition 3.4(ii)]{BHM} there exists $a \in R \setminus I$ with $\mathfrak{m}a \subseteq I$. Since $AI \subseteq I$, $M := Aa + I/I$ is a non-zero finitely generated $A$-submodule of $R/I$. Moreover $(\mathfrak{m}A)M = 0$ since $A$ is commutative, and since $A/\mathfrak{m}A$ is Artinian, so is the $A/\mathfrak{m}A$-module $M$. There are thus a simple left $A$-submodule $U$ of $M$ and a maximal ideal $N$ of $A$, with $\sum_{i=1}^s  A x_i \subseteq N$, such that $NU = 0.$ In particular, $N \setminus \sum_{i=1}^s A x_i$ consists of zero divisors on the left $A$-module $R/I$. That is, $\{x_1,\ldots , x_s \}$ is a maximal $A$-sequence on $R$ in $N$, so that, invoking (\ref{triad}) and (\ref{Krull}) for the second and third equalities,
$$ s = G_A(N,R) = \mathrm{ht}(N) = \mathrm{Kdim}(A).$$
However, this contradicts the hypothesis that the inequality in (\ref{keyer}) is strict. Hence equality holds in (\ref{keyer}). Since $\mathfrak{m}$ was an arbitrary maximal ideal of $C$, $R$ is a maximal Cohen-Macaulay left $C$-module.

\textbf{Step 3:} Let $\mathfrak{m}$ be a maximal ideal of $C$, and apply the Auslander-Buchsbaum formula \cite[Theorem 19.9]{E} to the regular local ring $C_{\mathfrak{m}}$ to get
\begin{equation}\label{AB} \mathrm{pr.dim}_{C_{\mathfrak{m}}}(C_{\mathfrak{m}} \otimes_C R) + G_{C_{\mathfrak{m}}}(\mathfrak{m}_{\mathfrak{m}}, C_{\mathfrak{m}} \otimes_C R) = G_{C_{\mathfrak{m}}}(\mathfrak{m}_{\mathfrak{m}}, C_{\mathfrak{m}}). \end{equation}
Now $G_C (\mathfrak{m},R) = \mathrm{ht}(\mathfrak{m})$ by Step 2. Thus, using also the preservation of $C$-sequences under localisation, \cite[Lemma 18.1]{E}, and the fact that $C$ is regular and hence Cohen-Macaulay,
$$ G_{C_{\mathfrak{m}}}(\mathfrak{m}_{\mathfrak{m}}, C_{\mathfrak{m}} \otimes_C R) = G_C (\mathfrak{m},R)= \mathrm{ht}(\mathfrak{m}) = G_C (\mathfrak{m},C) = G_{C_{\mathfrak{m}}}(\mathfrak{m}, C_{\mathfrak{m}}).$$
From this and (\ref{AB}) we deduce that $\mathrm{pr.dim}_{C_{\mathfrak{m}}}(C_{\mathfrak{m}} \otimes_C R)= 0.$ Since $\mathfrak{m}$ was an arbitrary maximal ideal of $C$, $R$ is a projective left $C$-module, as required.
\end{proof}

\begin{remarks}{\rm (i)} Theorem \ref{free}(ii) is false if the hypothesis that $R$ is equicodimensional is omitted. Consider the ring $R = k[X] \oplus k$ of Example \ref{baddy}, with the same subring $C = \{ (f(X), f(1)\} \cong k[X]$ as before. Clearly, $R$ is not a projective $C$-module.

{\rm(ii)} Theorem \ref{free}(ii) fails if $R$ is not a finitely generated $C$-module. For example, let $R$ be the coordinate ring of $SL(2, \mathbb{C})$, $R = \mathbb{C}[X,Y,Z,U]/\langle XU - YZ - 1 \rangle$. Let $C$ be the polynomial subalgebra $\mathbb{C}[X,Y]$ of $R$. Then $R$, being regular, is Cohen-Macaulay, but $R$ is not $C$-projective, since the maximal ideal $\mathfrak{m} = \langle X,Y \rangle$ of $C$ has $\mathfrak{m}R = R$.
\end{remarks}

\subsection{Azumaya and singular loci}\label{loci}
In this subsection we prove Theorem \ref{Az}, an improved version of \cite[Theorem 3.8]{BG}.
\begin{definition} Let $R$ be a noetherian ring which is a finite module over its centre $Z(R)$.

{\rm(i)} The \emph{Azumaya locus} of $R$ is
$$ \mathcal{A}_R = \{\mathfrak{m} \lhd_{\mathrm{max}} Z(R) : R_{\mathfrak{m}} \textit{ is Azumaya over } Z(R)_{\mathfrak{m}}\}.$$

{\rm(ii)} The \emph{smooth locus} of $Z(R)$ is
$$ \mathcal{F}_{Z(R)} = \{\mathfrak{m} \lhd_{\mathrm{max}} Z(R) : \mathrm{gl.dim}(Z(R)_{\mathfrak{m}}) < \infty \}. $$

\end{definition}

When $R$ as in the definition is in addition an affine $k$-algebra over an algebraically closed field $k$, then so is $Z(R)$ by the Artin-Tate lemma \cite[Lemma 13.9.10]{MR}, while if $R$ is prime then $Z(R)$ is a domain. In these circumstances $\mathcal{F}_{Z(R)}$ of course consists of the smooth points of the variety $\mathrm{Maxspec}(Z(R)),$ and both $\mathcal{A}_R$ and $\mathcal{F}_{Z(R)}$ are non-empty open subsets of $\mathrm{Maxspec}(Z(R))$. (For the case of $\mathcal{A}_{R}$, see for example \cite[Theorem III.1.7]{BG}.)

We have the following general relationship between the two sets defined above, following easily from the fact that, if $\mathfrak{m} \in \mathcal{A}_R$, then $R_{\mathfrak{m}}$ is a free $Z(R)_{\mathfrak{m}}$-module.

\begin{lemma} (\cite[Lemma 3.3]{BGhom}, \cite[Lemma III.1.8]{BG})\label{gotit} Let $R$ be a prime noetherian ring which is a finite module over its centre $Z(R)$. Suppose that $\mathrm{gl.dim.}(R) < \infty$. Then $\mathcal{A}_R \subseteq \mathcal{F}_{Z(R)}$.
\end{lemma}

 Of course, when $R = Z(R)$, $\mathcal{A}_R = \mathrm{Maxspec}(R)$, showing that the lemma's validity requires the finiteness of the global dimension of $R$. Note moreover that the inclusion given in the lemma is in general strict. Consider for example the enveloping algebra $U$ of the two-dimensional non-abelian Lie algebra over a field $k$ of positive characteristic $p$, $U = k \langle x,y :[y,x] = x \rangle$. This is a finite module over its centre $Z(U) = k \langle x^p, y^p - y \rangle$. By the PBW theorem $Z(U)$ is a polynomial algebra on the two given generators. Thus $\mathcal{F}_{Z(U)} = \mathrm{Maxspec}(Z(U))$. On the other hand, $xU$ is an ideal of $U$ with $U/xU \cong k[y]$, whence it follows easily that $\mathcal{A}_{Z(U)} = \mathrm{Maxspec}(Z(U)) \setminus \mathcal{V}(x^p)$.

As the following lemma and its corollary show, the problem for $U$ is that it has ``too many'' points which are not Azumaya: to be precise, the closed set of non-Azumaya primes of $Z(U)$ has codimension 1 in $\mathrm{maxspec}(Z(U))$, or - equivalently - there is a codimension one prime $\mathfrak{p}$ of $Z(U)$, namely $\langle x^p \rangle$, such that $U_{\mathfrak{p}}$ is not Azumaya. Thus we say that $R$ (a finite module over its noetherian centre $Z(R)$) is \emph{height 1 Azumaya} if the closed set $\mathrm{Maxspec}(Z(R)) \setminus \mathcal{A}_R$ has codimension at least 2.

\begin{lemma} \label{getout} Let $R$ be a prime noetherian ring, finitely generated and projective over
its centre $Z(R)$. If $R$ is height 1 Azumaya over $Z(R)$ then it is Azumaya over $Z(R)$.
\end{lemma}

For the proof of Lemma \ref{getout}, see \cite[Lemma 3.6]{BGhom}.

\begin{corollary} \label{opp} Let $R$ be a prime noetherian ring which is a finite module over
its centre Z(R). Suppose that $R$ is $Z(R)$-Macaulay. If $R$ is height 1 Azumaya over $Z(R)$, then
$\mathcal{F}_{Z(R)} \subseteq \mathcal{A}_R$.
\end{corollary}
\begin{proof} Let $\mathfrak{m} \in  \mathcal{F}_{Z(R)}$. Since $R$ is $Z(R)$-Macaulay, $R_{\mathfrak{m}}$ is $Z(R)_{\mathfrak{m}}$-Macaulay by Theorem \ref{CMprops}(iii). Since  $\mathfrak{m} \in  \mathcal{F}_{Z(R)}$, $$\mathrm{pr.dim.}_{Z(R)_{\mathfrak{m}}}(R_{\mathfrak{m}}) < \infty,$$
and so we can apply the  Auslander-Buchsbaum formula \cite[Theorem
19.9]{E} to deduce that $R_{\mathfrak{m}}$  is a projective $Z(R)_{\mathfrak{m}}$-module. Now Lemma \ref{getout} shows that $R_{\mathfrak{m}}$ is Azumaya over $Z(R)_{\mathfrak{m}}$, as required.
\end{proof}

From Lemma \ref{gotit} and Corollary \ref{opp} we deduce:

\begin{theorem}\label{Az} Let $R$ be a prime noetherian ring which is a finite module over its centre $Z(R)$. Consider the following statements:

{\rm(i)} $\mathrm{gl.dim}(R) < \infty$;

{\rm(ii)} $R$ is $Z(R)$-Macaulay;

{\rm(iii)} $R$ is height 1 Azumaya over $Z(R)$.

If (i) holds, then $\mathcal{A}_R \subseteq \mathcal{F}_{Z(R)}$, while if (ii) and (iii) hold then $\mathcal{F}_{Z(R)}  \subseteq \mathcal{A}_R$. Hence if (i),(ii) and (iii) hold,
$$\mathcal{A}_R = \mathcal{F}_{Z(R)}.$$
\end{theorem}

Theorem \ref{Az} improves \cite[Theorem 3.8]{BGhom}, since the latter requires the additional hypothesis that $R$ is Auslander-regular. Examples showing that this is a genuine improvement are given in $\S$\ref{recon}.

The discussion after Lemma \ref{gotit} shows that hypotheses (i) and (iii) are necessary in Theorem \ref{Az}. As regards hypothesis (ii), let $k$ be a field, $Z = k[X,Y]$, $\mathfrak{m} = \langle X,Y \rangle$, and set
$$   R   =  \begin{pmatrix} Z & \mathfrak{m} \\ Z & Z \end{pmatrix}.$$
Thus, being a finite module over its centre $Z$, $R$ is clearly noetherian, and it is easily seen to be prime. Being the idealiser of a maximal right ideal of $M_2(Z)$, $R$ has global dimension 2 by \cite[Corollary 7.5.12]{MR}. If $\mathfrak{p}$ is a codimension 1 prime of $Z$, then $\mathfrak{m} \cap (Z \setminus \mathfrak{p}) \neq \emptyset$, so that $R_{\mathfrak{p}} \cong M_2(Z_{\mathfrak{p}})$. Hence, $R$ is Azumaya in codimension 1. However, the conclusion of Theorem \ref{Az} is false for $R$, since $$ \mathcal{A}_R = \mathrm{Maxspec}(Z) \setminus \mathfrak{m} \qquad \subsetneq \qquad \mathrm{Maxspec}(Z) = \mathcal{F}_{Z(R)}.$$
It follows - as can easily be confirmed directly - that hypothesis (ii) fails for $R$.

\subsection{Reconstruction algebras: the $C$-Macaulay property} \label{recon} Wemyss \cite{W} introduced \emph{reconstruction algebras} in his work extending the concept of a noncommutative crepant resolution beyond its original application to Gorenstein surface singularities. In brief, let $r \geq 1$, let $\varepsilon$ be a primitive $r$th root of 1 in $\mathbb{C}$, let $a$ be a positive integer less than $r$ with $\mathrm{gcd}(r,a) = 1,$ and let $G = \langle \mathrm{diag}(\varepsilon, \varepsilon^a)\rangle$, a cyclic subgroup of order $r$ in $GL(2,\mathbb{C}).$ Thus $G$ acts linearly on $S := \mathbb{C}[X,Y]$, and we seek a (noncommutative) resolution of the invariant ring $Z :=  S^G.$ As
$\mathbb{C}G$-modules, $S = S_0 \oplus \cdots \oplus S_{r-1}$, the sum of the homogeneous components corresponding to the simple $\mathbb{C}G$-modules. By \cite[Corollary 10.10]{Yo}, the summands $S_i$ are precisely the indecomposable maximal Cohen-Macaulay $Z$-modules. One takes $M$ to be the sum of those $S_i$ which are \emph{special} in the sense of Wunram \cite{Wu}; the definition is recalled also in \cite[5.2.11]{W}. Then one sets the reconstruction algebra corresponding to the above data to be
\begin{equation}\label{construct} A_{r,a} \qquad := \qquad \mathrm{End}_{Z}(M).\end{equation}

\begin{theorem} \label{hardup} $A_{r,a}$ is $Z$-Macaulay.
\end{theorem}
\begin{proof}(Sketch) \textbf{Step 1:} \emph{The skew group algebra $\mathbb{C}[X,Y]\ast G$ is a maximal order.} This is a consequence of a general result hinging on the absence of pseudoreflections in the action of $G$ on $\mathbb{C}X \oplus \mathbb{C}Y$, \cite[Theorem 4.6]{M}.

\textbf{Step 2:}\; $ \mathrm{End}_Z(\mathbb{C}[X,Y]) \cong \mathbb{C}[X,Y]\ast G.$ Given Step 1, this is well-known, with details to be found in the proof of \cite[Theorem 1.5]{EG}, for example. First, note that the skew group algebra embeds in the endomorphism algebra, the element $\sum_{g \in G} s_g g$ mapping $t \in \mathbb{C}[X,Y]$ to $\sum_{g \in G}s_g g(t).$ To see that this embedding gives every endomorphism of $\mathbb{C}[X,Y]$, one checks that the isomorphism is valid after passing to quotient rings, that is after tensoring with $Q(Z)$, and then uses Step 1.

\textbf{Step 3:} \emph{Let $T$ be a prime $Z$-Macaulay ring and $e \in T$ a non-zero idempotent. Then $eTe$ is $eZe$-Macaulay, (with $eZe \cong Z$).} This is a routine check.

\textbf{Step 4:} The theorem follows from Steps 2 and 3, and the definition of $A_{r,a}$, letting $e$ be an idempotent in $T:= \mathbb{C}\ast G$ projecting from $\mathbb{C}[X,Y]$ onto $M$.
\end{proof}

\subsection{Reconstruction algebras: Azumaya locus}\label{reconAz}
\begin{theorem} With the notation of $\S$\ref{recon}, $A_{r,a}$ satisfies the hypotheses of Theorem \ref{Az}, with $Z(A_{r,a}) = Z$. In particular,
$$ \mathcal{A}_{A_{r,a}} = \mathcal{F}_{Z(A_{r,a})}.$$
\end{theorem}
\begin{proof} From the definition (\ref{construct}), $A_{r,a}$ is a factor algebra of $\mathrm{End}_Z(F)$ for a suitable finite rank free $Z$-module $F$, so that $A_{r,a}$ is a finite module over its central subalgebra $Z$. Since $M$ is $Z$-torsion free, so is $A_{r,a}$, so that we can embed $A_{r,a}$ in its quotient ring $Q(Z) \otimes_Z \mathrm{End}_Z(M) \cong \mathrm{End}_{Q(Z)}(Q(Z) \otimes M)).$ Since the latter algebra is prime, so is $A_{r,a}$.

The global dimension of $A_{r,a}$ is finite by \cite[Theorem 6.18]{W}, so (i) of Theorem \ref{Az} holds for $A_{r,a}$. Hypothesis (ii) of Theorem \ref{Az} is guaranteed by Theorem \ref{hardup}. Finally, let $\mathfrak{p}$ be a codimension one prime of $Z$, and let $m = \mathrm{dim}_{Q(Z)}(Q(Z) \otimes M).$ Since $Z$ is by definition an invariant ring, it is integrally closed \cite[Proposition 6.4.1]{BH}. Thus $Z_{\mathfrak{p}}$ is a DVR, and so
$$( A_{r,a})_{\mathfrak{p}} = \mathrm{End}_Z (M) \otimes_Z Z_{\mathfrak{p}} \cong  \mathrm{End}_{Z_{\mathfrak{p}}}(M_{\mathfrak{p}}) \cong M_m(Z_{\mathfrak{p}}). $$
Therefore $A_{r,a}$ is Azumaya in codimension 1, giving hypothesis (iii) of Theorem \ref{recon}.
\end{proof}

\section{Grade symmetry}\label{jsymm}
The purpose of this section is to define grade symmetric $C$-Macaulay rings, and describe their basic properties. We propose that grade symmetry gives the correct strengthening of the $C$-Macaulay condition, so as to retrieve for rings finite over their centres the familiar hierarchy of commutative homological properties - Cohen-Macaulay implies Gorenstein implies regular. The noncommutative version of the hierarchy is established in $\S$\ref{homhier}.
\subsection{Homological grade} \label{homdef}
\begin{definition} \label{symdef}{\rm(i)} Let $R$ be a ring and let $M$ be a non-zero right $R$-module. The \emph{right homological grade} of $M$ is
$$ j_R^r(M) := \mathrm{min}\{i: \mathrm{Ext}_R^i(M, R_R) \neq 0 \}, $$
or $j_R^r(M) = \infty$ if no such $i$ exists. The left homological grade of a left $R$-module $N$, denoted $j_R^{\ell}(N)$, is defined similarly. The suffix and superfix decorating $j$ will be omitted whenever possible.

{\rm(ii)} An $R$-bimodule $M$ is \emph{central} if $zm = mz$ for all $m \in M$ and $z \in Z(R).$

{\rm(iii)} $R$ is \emph{grade symmetric} if $j_R^{\ell}(M) = j_R^{r}(M)$ for all central $R$-bimodules $M$.
\end{definition}

Grade symmetry was first defined and studied in \cite{ASZ}, mainly in the context of Auslander-Gorenstein rings, and without the restriction to \emph{central} bimodules imposed above.

The concept of an \emph{exact dimension function} $\delta_{R}$ on the modules of a ring $R$ can be found in, for example \cite[6.8.4]{MR}, or \cite{Lev}. A dimension function $\delta_R$, defined on both left and right $R$-modules, is \emph{symmetric} if, for every central $R$-bimodule $M$, finitely generated on both sides, $\delta_R (M)$ takes the same value whether $M$ is viewed as a left module or as a right module.

\begin{definition} \label{deltamac} Let $\delta$ be an exact dimension function for the noetherian ring $R$.

{\rm (i)} $R$ is $\delta$-\emph{Macaulay} if
\begin{equation}\label{delmac} \delta (R) = \delta (M) + j_R(M) \end{equation}
for all finitely generated right or left $R$-modules $M$.

{\rm (ii)} If $R$ is finite over its centre $Z(R)$ and, for every maximal ideal $\mathfrak{m}$ of $Z(R)$,
\begin{equation} \label{delmacloc} \delta (R_{\mathfrak{m}}) = \delta (M) + j_{R_{\mathfrak{m}}}(M) \end{equation}
for all finitely generated right or left $R_{\mathfrak{m}}$-modules $M$, then we say that $R$ is \emph{locally $\delta$-Macaulay}.
\end{definition}

Of course it's implicit in (ii) of the definition that $\delta$ is also defined for $R_{\mathfrak{m}}$-modules, for all maximal ideals $\mathfrak{m}$ of $Z(R)$.

\begin{rexamples}\label{graderex}{\rm(i)} It is trivial but nonetheless important to observe that if a ring $R$ is $\delta$-Macaulay for a symmetric dimension function $\delta$, then it is grade symmetric.

{\rm(ii)} Two standard choices for $\delta$ are the \emph{Krull dimension}, which we denote by $\mathrm{Kdim}_R$, and the \emph{Gelfand-Kirillov dimension}, $\mathrm{GKdim}_R$; for details see, for example, \cite[Chapters 6 and 8]{MR}, \cite{KL}. Note that, while $\mathrm{Kdim}_R$ is defined for every noetherian ring $R$, it is not known whether it is always symmetric; $\mathrm{Kdim}_R$ is, however, easily seen to be symmetric when $R$ is finite over a central subring, \cite[Corollary 6.4.13]{MR}. On the other hand, $\mathrm{GKdim}_R$ is not defined for every $k$-algebra $R$, but, when it is defined, it is always symmetric, \cite[Corollary 5.4]{KL}. For an affine $k$-algebra $R$ which is finite over a central subring, $\mathrm{Kdim}_R$ coincides with $\mathrm{GKdim}_R$; this is an easy consequence of the equality of the two dimensions for affine commutative algebras, \cite[Theorem 4.5]{KL}.

{\rm(iii)} If (\ref{delmac}) holds for $R$ when $\delta$ is either of the two cases featuring in \ref{graderex}(ii), we say that $R$ is \emph{Krull Macaulay} or \emph{GK-Macaulay} respectively. By \cite[Corollary 2.1.4]{BH},
\begin{equation}\label{comm} \textit{commutative local Cohen-Macaulay rings are Krull-Macaulay.}
\end{equation}
The converse of (\ref{comm}) is also true, and well-known, though we have not been able to locate a reference. Both directions are special cases of Theorem \ref{opus}(i)$\Leftrightarrow$(iv) below.

{\rm(iv)} As discussed in \cite[$\S$4.5]{Lev}, one can reverse the order of development of Definition \ref{deltamac}, and ask for conditions under which, for a given ring $R$, there exists a non-negative integer $n$ such that $\delta_R := n - j_R$ defines an exact dimension function for $R$. For example, this is possible when $R$ is an Auslander-Gorenstein ring by \cite[Proposition 4.5]{Lev}.

{\rm(v)} Let $k$ be a field of characteristic 0, let $n$ be a positive integer, and let $R := A_n (k)$ be the $n$th Weyl algebra over $k$. Then $\mathrm{GKdim}(R) = 2n$ and $R$ is GK-Macaulay by \cite[Chapter 2, $\S$7.1]{Bj}. But $R$ is \emph{not} Krull Macaulay when $n > 1$: by a result of Stafford \cite{S}, $R$ has a principal maximal left ideal $I$, so, using \cite[Theorem 6.6.15]{MR}, we easily calculate that
$$ \mathrm{Kdim}_R(R) = n > 1 = 0 + 1 = \mathrm{Kdim}_R (R/I) + j_R^{\ell}(R/I).$$

{\rm(vi)} Let $k$ be a field and let $R$ be the algebra of $2 \times 2$ upper triangular matrices over $k$, with maximal ideal $I := \begin{pmatrix} k & k \\ 0 & 0 \end{pmatrix}$. Trivially, by Examples \ref{CMex}(ii), $R$ is $k$-Macaulay. One easily checks that $j_R^{\ell}(R/I) = 1$ and $j^r_R(R/I) = 0$, so that $R$ is neither Krull-Macaulay, nor GK-Macaulay, nor grade symmetric.
\end{rexamples}

\begin{lemma} \label{KCMtest} Let $R$ be a noetherian ring which is a finite module over its centre $Z(R)$. Then $R$ is Krull-Macaulay $\Leftrightarrow$ for every prime ideal $P$ of $R$,
\begin{equation}\label{primemac} \mathrm{Kdim}(R) = \mathrm{Kdim}(R/P) + j_R(R/P). \end{equation}
\end{lemma}
\begin{proof} $\Rightarrow$: Trivial

$\Leftarrow$: Let $M$ be a finitely generated $R$-module. We prove by induction on $\mathrm{Kdim}(M)$ that
\begin{equation} \label{key} n := \mathrm{Kdim}(R) = \mathrm{Kdim}(M) + j_R(M). \end{equation}
$\mathrm{Kdim}(M) = 0:$ Let $X$ be an irreducible $R$-module. Then $R/\mathrm{Ann}_R(X)$ is simple Artinian by Kaplansky's theorem, \cite[Theorem 3.3.8]{MR}, so $R/\mathrm{Ann}_R(X)$ is a finite direct sum of copies of $X$, and (\ref{key}) for $M=X$ follows from (\ref{primemac}) for $P = \mathrm{Ann}_R(X)$. Now (\ref{key}) for an arbitrary finitely generated artinian $R$-module follows by a routine application of the long exact sequence of cohomology.

For the induction step, let $t \geq 0$ and suppose that (\ref{key}) is proved when $\mathrm{Kdim}(M) \leq t$. Assume that $\mathrm{Kdim}(M) = t+1$. By \cite[Theorem 9.6]{GW}, since $R$ is FBN, $M$ has a generalised composition series
$$ 0 = M_0 \subseteq M_1 \subset M_2 \subset \cdots \subset M_r = M, $$
where $r \geq 2$, $\mathrm{Kdim}(M_1) \leq t$ and $0 \neq M_j/M_{j-1} \cong U_j$, with $U_j$ a uniform left ideal of $R/P_j$ and $P_j$ a prime ideal of $R$ with $\mathrm{Kdim}(R/P_j) = t+1$, $1 < j \leq r$.

Suppose that (\ref{key}) is proved with $M$ replaced by any subfactor in the above generalised composition series. So $j(M_1) \geq n-t$ and $j(M_j/M_{j-1}) = n-t-1$, for $1<j \leq r$. Therefore, for $\ell < n-t-1$, the long exact sequence of $\mathrm{Ext}$ shows that
$$ \mathrm{Ext}^{\ell}_R(M,R) = 0, $$
and hence $j(M) \geq n-t-1$. Then, the short exact sequence
$$ 0 \longrightarrow M_{r-1}\longrightarrow M \longrightarrow M/M_{r-1} \cong U_r \longrightarrow 0$$
yields
$$ 0 = \mathrm{Ext}_R^{n-t-2}(M_{r-1}, R) \longrightarrow\mathrm{Ext}_R^{n-t-1}(U_r, R) \longrightarrow \mathrm{Ext}_R^{n-t-1}(M, R).$$
Since the middle term is non-zero, $j(M) = n-(t+1)$, as required.

So we have to prove (\ref{key}) when $M$ is a uniform left ideal of a prime factor ring $R/P$ of $R$, with $\mathrm{Kdim}(R/P) = t+1$. Let $m$ be the Goldie rank of $R/P$, and set $I$ to denote the direct sum of $m$ copies of $M$. By \cite[7.24 and 7.8(c)]{GW}, there are exact sequences
\begin{equation}\label{twee} 0\longrightarrow I \longrightarrow R/P \longrightarrow X \longrightarrow 0 \end{equation}
and
\begin{equation}\label{dee} 0\longrightarrow R/P \longrightarrow I \longrightarrow Y \longrightarrow 0 \end{equation}
with $\mathrm{Kdim}(X) \leq t$ and $\mathrm{Kdim}(Y) \leq t$. By (\ref{primemac}),
\begin{equation} \label{hind} j(R/P) = n-(t+1) \end{equation}
and, by induction,
\begin{equation} \label{sight} j(X) \geq n-t \textit{  and  }  j(Y) \geq n-t. \end{equation}
Let $w < n-(t+1)$. Then (\ref{dee}) gives the exact sequence
$$ \mathrm{Ext}^w_R(Y,R)\longrightarrow  \mathrm{Ext}^w_R(I,R)\longrightarrow  \mathrm{Ext}^w_R(R/P,R), $$
where the outside terms are both zero by (\ref{hind}) and (\ref{sight}). Hence
\begin{equation}\label{glory} j_R(I) \geq n-(t+1). \end{equation}
Then (\ref{twee}) gives
 \begin{equation}  \mathrm{Ext}_R^{n-(t+1)}(X,R)\longrightarrow \mathrm{Ext}_R^{n-(t+1)}(R/P,R)\longrightarrow \mathrm{Ext}_R^{n-(t+1)}(I,R), \end{equation}
 where the first term is zero by (\ref{sight}) and the second term is non-zero by (\ref{hind}). Hence,
 \begin{equation}\label{here} j_R(I) \leq n-(t+1). \end{equation}
 Thus the induction step follows from (\ref{glory}) and (\ref{here}).
\end{proof}

The connection between homological grade and grade defined in terms of $R$-sequences is afforded by the following lemma.

\begin{lemma} \label{grades} Let $R$ be a noetherian ring which is a finite module over its centre $Z(R)$. Let $P$ be a prime ideal of $R$ and let $\mathfrak{p} = P \cap Z(R)$.

{\rm(i)} $G(\mathfrak{p},R) \leq \mathrm{min}\{j_R^{\ell}(R/P), j_R^r(R/P)\}$.

{\rm(ii)} Suppose that $P$ is a maximal ideal of $R$. There exists a maximal ideal  $Q$ of $R$ with $Q \cap Z(R) = \mathfrak{p}$ and $G(\mathfrak{p},R) = j_R^{\ell}(R/Q).$ A similar statement holds for $j_R^r.$

For the remainder of the lemma, assume that $R$ is $Z(R)$-Macaulay.

{\rm(iii)} There exists a prime ideal  $Q$ of $R$ with $Q \cap Z(R) = \mathfrak{p}$ and $G(\mathfrak{p},R) = j_R^{\ell}(R/Q).$ A similar statement holds for $j_R^r.$

{\rm(iv)} Let $Q$ be any prime ideal of $R$ for which $t := G(\mathfrak{q},R) = j_R^{\ell}(R/Q)$. Then $\mathrm{Ext}_R^t(R/Q, R)$ is a torsion-free $Z(R)/\mathfrak{q}$-module.

{\rm(v)} Let $Q$ be  a prime ideal of $R$, with $\mathfrak{q} := Q \cap R$, and $t := G(\mathfrak{q},R)$. Let $B$ be a non-zero finitely generated torsion-free $R/Q$-module. Then $j_R^{\ell}(B) = t$ if and only if $j_R^{\ell}(R/Q) = t.$ A similar statement applies to $j_R^r$ when $B$ is a right module.
\end{lemma}

\begin{proof} Let $P$ and $\mathfrak{p}$ be as stated.

(i) Let $\{x_1, \ldots, x_t\}$ be a maximal $Z(R)$-sequence in $P$ and set $I := \sum_i x_iR.$ By the Change of Rings Theorem \cite[Theorem 9.37]{R},
\begin{equation}\label{change} \mathrm{Ext}_{R/I}^m(R/P,R/I) \cong \mathrm{Ext}_R^{m+t}(R/P,R), \end{equation}
for all $m \in \mathbb{Z}.$ This proves (i).

(ii) Suppose now that $P$ is maximal. Keep the notation as above. The maximality of $t$ implies that $\mathfrak{p}/I$ consists of zero divisors in $R/I$, so, by \cite[Proposition 3.4]{BHM}, there exists $y \in R \setminus I$ such that $\mathfrak{p}y \subseteq I.$ Thus there exists a left annihilator prime ideal $Q/I$ of $R/I$ with $\mathfrak{p} \subseteq Q,$ and maximality of $\mathfrak{p}$ ensures that $Q\cap Z(R) = \mathfrak{p}$.  Thus, working with \emph{left} modules, $\mathrm{Hom}_{R}(R/Q, R/I) \neq 0$, and therefore (\ref{change}) with $P$ replaced by $Q$ shows that $j_R^{\ell}(R/Q) \leq t.$ Combining this inequality with (i) for $Q$ in place of $P$ now gives (ii), since by INC, Proposition \ref{INCetc}(ii), $Q$ is a maximal ideal of $R$.

Suppose henceforth that $R$ is $Z(R)$-Macaulay.

{\rm(iii)}  Carry through the proof as in (ii). Everything proceeds in the same way to get a prime ideal $Q$ of $R$ with $I \subseteq Q$ and $\mathfrak{p} \subseteq Q \cap Z(R)$, with $Q/I$ a left annihilator prime of $R/I$. By (\ref{change}),
\begin{equation}\label{guy} j_R^{\ell}(R/Q) = G(\mathfrak{p},R). \end{equation}
Let $\mathfrak{q} := Q \cap R$, so that $\mathfrak{p} \subseteq \mathfrak{q}.$ Suppose for a contradiction that this inclusion is strict, so that $\mathrm{ht}(\mathfrak{p}) < \mathrm{ht}(\mathfrak{q}).$ Since $R$ is $Z(R)$-Macaulay, Theorem \ref{CMprops}(ii) and (\ref{guy}) imply that
$$  G(\mathfrak{q},R) =  \mathrm{ht}(\mathfrak{q}) > \mathrm{ht}(\mathfrak{p}) = G(\mathfrak{p},R) = j_R^{\ell}(R/Q).$$
But this contradicts (i) of the lemma, so $Q \cap R = \mathfrak{p}$ as required.

{\rm(iv)} Let $Q$ be a prime for which $G(\mathfrak{q}, R) = j_R^{\ell}(R/Q) := t$. Let $\{x_1, \ldots, x_t\}$ be a maximal $Z(R)$-sequence in $Q$ and set $I := \sum_i x_iR.$ By (\ref{change}), $ \{0\} \neq \mathrm{Ext}^t_R(R/Q,R) \cong \mathrm{Hom}_{R/I}(R/Q,R/I).$ Suppose that $f$ is a non-zero element of $\mathrm{Hom}_{R/I}(R/Q,R/I)$ with $fz = 0$ for some $z \in Z(R) \setminus \mathfrak{q}$. So $\mathrm{im}(f)$ is a non-zero left ideal of $R/I$ with $(Q + Rz)\mathrm{im}(f) = 0$, and hence there is an annihilator prime $W/I$ of $R/I$ with $Q/I \varsubsetneq W/I$. By Small's theorem, \cite[Corollary 4.1.4]{MR}, this contradicts Theorem \ref{CMprops}(vi) applied to $R/I$, since $R/I$ is $Z(R)+I/I$-Macaulay by Theorem \ref{CMprops}(vii). This proves (iv).

{\rm(v)} Let $Q$ and $B$ be as stated, with $\{x_1, \ldots, x_t\}$ be a maximal $Z(R)$-sequence in $Q$, $I := \sum_i Rx_i.$ Suppose first that $j_R^{\ell}(R/Q) = t.$ As in the proof of (i), $\mathrm{Ext}_{R/I}^m(B,R/I) \cong \mathrm{Ext}_R^{m+t}(B,R)$ for all $m \in \mathbb{Z}$. Hence, $j_R^{\ell}(B) \geq t.$ To prove the reverse inequality, localise $\mathrm{Ext}_{R}^t(R/Q,R)$ by inverting $Z(R) \setminus \mathfrak{q}$: we see from (iv) that $\mathrm{Ext}_{R_{\mathfrak{q}}}^{t}(R_{\mathfrak{q}}/Q_{\mathfrak{q}},R_{\mathfrak{q}})$ is non-zero. Since this space is the direct sum of copies of $\mathrm{Ext}_{R_{\mathfrak{q}}}^{t}(V,R_{\mathfrak{q}})$, where $V$ is the irreducible left $R_{\mathfrak{q}}/Q_{\mathfrak{q}}$-module, it follows that $\mathrm{Ext}_{R_{\mathfrak{q}}}^{t}(V,R_{\mathfrak{q}})$ is non-zero also. But the localisation of $\mathrm{Ext}_R^{t}(B,R)$ at $\mathfrak{q}$, that is $\mathrm{Ext}_{R_{\mathfrak{q}}}^t(B_{\mathfrak{q}}, R_{\mathfrak{q}}),$ is the direct sum of a non-zero number of copies of $\mathrm{Ext}_{R_{\mathfrak{q}}}^{t}(V,R_{\mathfrak{q}})$, and is therefore non-zero. It follows that $j_R^{\ell}(B) \leq t$, and so we get equality. The converse, and the proof for $j_R^r$ are exactly similar.
\end{proof}

Example \ref{graderex}(vi) shows that we can't always take $Q = P$ in (ii) or (iii) of the above lemma, and indeed that different ideals $Q$ may be needed for $j^{\ell}$ and $j^r$.

\subsection{Basic properties of grade symmetric $Z(R)$-Macaulay rings} \label{equiv} A key feature of the grade symmetry property is encapsulated in the following consequence of M$\ddot{\mathrm{u}}$ller's theorem, \cite[Theorem III.9.2]{BG}, \cite{Mu}.

\begin{lemma} \label{muller} Let $R$ be a noetherian ring which is a finite module over its centre $Z(R)$.

{\rm(i)} Suppose that $R$ is grade symmetric. Let $P$ be a maximal ideal of $R$, with $ \mathfrak{p} := P \cap Z(R).$ Then
\begin{equation} \label{lockin} j_R(R/P) = G(\mathfrak{p},R).\end{equation}

{\rm(ii)} If $R$ is grade symmetric and $Z(R)$-Macaulay, then (\ref{lockin}) is true for all prime ideals of $R$. In particular, for all primes $P$ of $R$, setting $\mathfrak{p} := P \cap Z(R),$
$$ j_R(R/P) = \mathrm{ht}(\mathfrak{p}), $$
and $\mathrm{Ext}_R^{j(R/P)}(R/P,R)$ is a torsion-free $Z(R)/\mathfrak{p}$-module.
\end{lemma}
\begin{proof} (i) Suppose that $R$ is grade symmetric, and let $P$ be a maximal ideal of $R$ with $\mathfrak{p} := P \cap Z(R).$ Lemma \ref{grades}(ii) guarantees the existence of a maximal ideal $Q$ of $R$ with $Q \cap Z(R) = \mathfrak{p}$ and $j_R^{\ell}(R/Q) = G(\mathfrak{p}, R)$. By M$\ddot{\textrm{u}}$ller's theorem, \cite{Mu}, see also \cite[Theorem III.9.2]{BG}, there is an integer $n \geq 1$ and a finite sequence of non-zero finitely generated central $R/P_i - R/P_{(i+1)}$-bimodules $B_1, \ldots , B_n$, with $Q = P_1, \ldots , P_n = P$, such that each $P_i$ is a maximal ideal of $R$ whose intersection with $Z(R)$ is $\mathfrak{p}$. From grade symmetry, bearing in mind also that each factor $R/P_i$ is simple artinian since $\mathfrak{p}$ is a maximal ideal of $Z(R)$, it follows that $j_R^r (R/P) = j_R^{\ell}(R/P) = j_R^{\ell}(R/Q)$.

(ii) The proof is exactly the same as the proof of (i), except that now we need to invoke Lemma \ref{grades}(v) to handle the homological grade of the bimodules $B_i$, and Lemma \ref{grades}(iv) gives the final claim.
\end{proof}

For injectively homogeneous rings, (see Definition \ref{injhom}), a more precise version of the following result was proved in \cite[Theorem 5.5]{BrH2}, generalised in \cite[Theorem 4.2]{ASZ} to grade symmetric Auslander-Gorenstein rings satisfying a polynomial identity. The final sentence of the corollary was proved for noetherian rings of finite injective dimension in \cite[Corollary 4.6]{Mi}; see also \cite[Theorem 2.3]{ASZ}.

\begin{corollary} \label{resolve} Let $R$ be a noetherian ring which is a finite module over its centre $Z(R)$, and suppose that $R$ is grade symmetric and $Z(R)$-Macaulay. Let
\begin{equation}\label{res} 0 \rightarrow R \rightarrow E_0 \rightarrow E_1 \rightarrow \cdots \rightarrow E_i \rightarrow \cdots \end{equation}
be a minimal injective resolution of $\,_RR$, and let $P$ be a prime ideal of $R$, with $\mathrm{ht}(\mathfrak{p}) = t$. Then the first occurrence in (\ref{res}) of the indecomposable injective $R$-module with assassinator $P$ is as a summand of $E_t$. In particular, every indecomposable injective occurs at least once in (\ref{res}).
\end{corollary}
\begin{proof} This is immediate from Lemma \ref{muller}(ii) together with the use of (\ref{res}) to calculate $\mathrm{Ext}^*_R(R/P,R)$. (For the latter, see for example \cite[Lemma 2.3]{B}.)
\end{proof}

\begin{theorem} \label{opus} Let $R$ be a noetherian ring which is a finite module over its centre $Z(R)$. Then the following statements are equivalent:

{\rm(i)} $R$ is Krull-Macaulay;

{\rm(ii)} $R$ is equicodimensional and locally Krull-Macaulay;

{\rm(iii)} $R$ is $Z(R)$-Macaulay, equicodimensional and grade symmetric;

{\rm(iv)} $G_{Z(R)}(\mathfrak{m}, R) = \mathrm{Kdim}(R)$ for all maximal ideals $\mathfrak{m}$ of $Z(R)$, and $R$ is grade symmetric.

\end{theorem}

\begin{proof} (i)$\Rightarrow$(iv): Assume that $R$ is Krull-Macaulay. Then $R$ is grade symmetric by Remark \ref{graderex}(ii). Let $\mathfrak{m}$ be a maximal ideal of $Z(R)$. By Lemma \ref{grades}(ii) there exists a maximal ideal $Q$ of $R$ with $Q \cap Z(R) = \mathfrak{m}$, for which
\begin{equation} \label{gumbo} G_{Z(R)}(\mathfrak{m},R) = j_R^{\ell}(R/Q).\end{equation}
Since $\mathrm{Kdim}(R/Q) = 0$, the Krull Macaulay property now ensures that $G_{Z(R)}(\mathfrak{m}, R) = \mathrm{Kdim}(R)$.

(iv)$\Rightarrow$(iii): Assume (iv). Let $M$ be a maximal ideal of $R$. Then
$$ G_{Z(R)}(\mathfrak{m},R) \leq \mathrm{ht}(M) \leq \mathrm{Kdim}(R), $$
by Theorem \ref{CMprops}(i), so (iv) forces equality throughout. The first equality implies that $R$ is $Z(R)$-Macaulay by Theorem \ref{CMprops}(v), and the second gives equicodimensionality.

(iii)$\Rightarrow$(ii): Suppose that $R$ is $Z(R)$-Macaulay, equicodimensional and grade symmetric. Let $P$ be a prime ideal of $R$, with $\mathfrak{p} := P \cap Z(R).$ By Lemma \ref{muller}(ii),
\begin{equation} \label{mullit} j_R(R/P) = G(\mathfrak{p},R)=: t.\end{equation}
 Thus, by Lemma \ref{grades}(iv), $\mathrm{Ext}^t_R(R/P, R)$ is non-zero and $Z(R)/\mathfrak{p}$-torsion-free. Hence, if $\mathfrak{m}$ is any maximal ideal of $Z(R)$ with $\mathfrak{p} \subseteq \mathfrak{m}$,
\begin{equation}\label{getit} j_{R_{\mathfrak{m}}}(R_{\mathfrak{m}}/P_{\mathfrak{m}}) = t. \end{equation}
Since $R$ is $Z(R)$-Macaulay, $R_{\mathfrak{m}}$ is $Z(R_{\mathfrak{m}})$-Macaulay by Theorem \ref{CMprops}(iv), and so
\begin{equation} \label{hike} G(\mathfrak{p},R) = \mathrm{ht}(P) = \mathrm{ht}(P_{\mathfrak{m}}) = G(P_{\mathfrak{m}}, R_{\mathfrak{m}}). \end{equation}
 It follows from (\ref{mullit}), (\ref{getit}) and (\ref{hike}) that $\mathrm{ht}(P_{\mathfrak{m}}) = j_{R_{\mathfrak{m}}}(R_{\mathfrak{m}}/P_{\mathfrak{m}})$. Since Theorem \ref{catenary}(ii) ensures that $\mathrm{Kdim}(R_{\mathfrak{m}}) = \mathrm{Kdim}(R_{\mathfrak{m}}/P_{\mathfrak{m}}) + \mathrm{ht}(P_{\mathfrak{m}})$, we obtain (\ref{delmacloc}) for $M = R/P$ and for $\mathfrak{m}$. Thus (ii) follows from Lemma \ref{KCMtest}.

 (ii)$\Rightarrow$(i): Suppose that $R$ is equicodimensional and locally Krull Macaulay. By Lemma \ref{KCMtest}, we must show that, for every prime $P$ of $R$,
 \begin{equation} \label{fate} \mathrm{Kdim}(R) = \mathrm{Kdim}(R/P) + j_R(R/P),
 \end{equation}
 where $j_R$ denotes both the left and the right grade. Let $\mathfrak{m}$ be a maximal ideal of $Z(R)$. By the already-proved ${\rm(i)}\Rightarrow{\rm(iv)}\Rightarrow{\rm(iii)}$ of the theorem applied to the Krull Macaulay ring $R_{\mathfrak{m}}$, we see that $R_{\mathfrak{m}}$ is $Z(R)_{\mathfrak{m}}$-Macaulay. Since $\mathfrak{m}$ was arbitrary, Theorem \ref{CMprops}(iv) ensures that $R$ is $Z(R)$-Macaulay. Bearing in mind that $R$ is also by hypothesis equicodimensional, Theorem \ref{chains}(iii) implies that, for all primes $P$ of $R$,
 \begin{equation}\label{cravat} \mathrm{Kdim}(R) = \mathrm{Kdim}(R/P) + \mathrm{ht}(P). \end{equation}
 By (\ref{cravat}), (\ref{fate}) holds if and only if,  for all primes $P$ of $R$,
 \begin{equation} \label{cod} j_R (R/P) = \mathrm{ht}(P). \end{equation}
To prove (\ref{cod}) for $j_R^{\ell}$, note first that
\begin{equation}\label{sole}j_R^{\ell} (R/P) \geq G_{Z(R)}(\mathfrak{p},R) = \mathrm{ht}(P) = \mathrm{ht}(\mathfrak{p}), \end{equation}
by Theorem \ref{CMprops}(ii) and Lemma \ref{grades}(i). For the opposite inequality, by Lemma \ref{grades}(iii) there is a prime $Q$ of $R$ with $Q \cap Z(R) = \mathfrak{p}$, such that
\begin{equation} \label{lone} t:= j_R^{\ell}(R/Q) = G_{Z(R)}(\mathfrak{p}, R) = \mathrm{ht} (Q) = \mathrm{ht}(\mathfrak{p}), \end{equation}
where for the final two equalities we again use Theorem \ref{CMprops}(ii), applicable since $R$ is $Z(R)$-Macaulay.
By Lemma \ref{grades}(iv) $\mathrm{Ext}^t_R (R/Q,R)$ is a torsion-free $Z(R)/\mathfrak{q}$-module. Thus, letting $\mathfrak{m}$ be any maximal ideal of $Z(R)$ which contains $\mathfrak{p}$, we deduce that $j^{\ell}_{R_{\mathfrak{m}}}(R_{\mathfrak{m}}/Q_{\mathfrak{m}}) = t$. Now local grade symmetry, which holds thanks to hypothesis (ii), forces $j^{\ell}_{R_{\mathfrak{m}}}(R_{\mathfrak{m}}/P_{\mathfrak{m}}) = t$, by Lemma \ref{muller}(ii). Therefore
\begin{equation} \label{cinch} j_R^{\ell}(R/P) \leq t. \end{equation}
Now (\ref{cod}) for $j_R^{\ell}$ follows from (\ref{sole}), (\ref{cinch}) and (\ref{lone}). The argument for $j_R^r$ is identical.
\end{proof}

We now show that the characterisation in Lemma \ref{muller}(ii) of $j_R (R/P)$, obtained there for a prime ideal $P$ of a grade symmetric $Z(R)$-Macaulay ring, extends naturally to $j_R(X)$ for all finitely generated $R$-modules $X$.

\begin{theorem}\label{gradequiv}

Let $R$ be a ring which is a finite module over its noetherian centre $Z(R)$. Suppose that $R$ is $Z(R)$-Macaulay. Then the following are equivalent.

{\rm(i)} $R$ is grade symmetric.

{\rm(ii)} For all prime ideals $P$ of $R$, $j_R^{\ell}(R/P) = j_R^{r}(R/P) = G_{Z(R)}(\mathfrak{p}, R)$.

{\rm(iii)} For every finitely generated left $R$-module $X$,
\begin{equation}\label{value} j_R^{\ell}(X) = \mathrm{inf}\{G_{Z(R)}(\mathfrak{p}, R) : \mathfrak{p} \textit{ an annihilator prime of } \,_{Z(R)}X\};\end{equation}
and similarly for finitely generated right $R$-modules.

{\rm(iv)} For every prime ideal $P$ of $R$, and for some (equivalently, for every) maximal $Z(R)$-sequence $\{x_1, \ldots , x_t \}$ in $P$, $P$ is a left and a right annihilator prime of $R/\sum_{i}x_iR$.

\end{theorem}

\begin{proof} {\rm(iii)}$ \Rightarrow${\rm(i)}: This is clear since grade symmetry is defined with respect to central bimodules.

{\rm(i)}$ \Rightarrow ${\rm(ii)}: Lemma \ref{muller}(ii) and Theorem \ref{CMprops}(ii).

{\rm(ii)}$ \Rightarrow${\rm(iii)}: Suppose that (ii) holds. We prove (iii) for \emph{left} $R$-modules $X$; the argument on the right is parallel. We argue by induction on the Krull dimension $\mathrm{Kdim}(X)$ of the module $X$. Suppose initially that $X$ is artinian as $R$-module, and hence also as $Z(R)$-module. Amongst the annihilator primes of $X$ as $Z(R)$-module, which are all maximal ideals of $Z(R)$, let $\mathfrak{m}$ be one of minimal codimension, $t$ say. Since $R$ is $Z(R)$-Macaulay, we have to show that $j_R^{\ell}(X) = t.$ For this, we induct on the composition length $m$ of $X$, the case $m = 1$ being given by (ii). By standard properties of modules over commutative artinian rings, $\mathfrak{m}X \subsetneqq X$, and so there is an exact sequence of $R$-modules
$$ 0 \rightarrow B\rightarrow X \rightarrow Y \rightarrow 0, $$
with $Y$ a simple $R$-module with $\mathfrak{m}Y = 0.$ For $i < t$ there is an exact sequence
$$ \mathrm{Ext}^i(Y,R)\rightarrow \mathrm{Ext}^i(X,R)\rightarrow \mathrm{Ext}^i(B,R) $$
whose two outer terms are 0, by (ii) and induction on $m$ respectively. Thus $j_R^{\ell}(X) \geq t.$ Now consider the exact sequence
$$ \mathrm{Ext}^{t-1}(B,R)\rightarrow\mathrm{Ext}^t(Y,R)\rightarrow\mathrm{Ext}^t(X,R).$$
Since, by choice of $Y$, $\mathrm{ht}(\mathfrak{n}) \geq \mathrm{ht}(\mathfrak{m})$ for all annihilator primes $\mathfrak{n}$ of $B$, the left-hand term above is 0 by the induction on $m$. The middle term is non-zero by (ii), so $\mathrm{Ext}_R^t(X,R) \neq 0$, proving the $m$-induction step for the artinian case of (iii).

Now suppose that $r > 0$, that $X$ has Krull dimension $r$, and that (iii) has been proved for all finitely generated $R$-modules whose Krull dimension is less than $r$. We consider first a special case: namely, let $P$ be a prime ideal of $R$ with $\mathrm{Kdim}(R/P) \leq r$, and let $U$ be a uniform left ideal of $R/P$. Then we claim that
\begin{equation}\label{uniform}   j_R^{\ell}(U) = G_{Z(R)}(\mathfrak{p},R) = t.
\end{equation}

Let $s$ be the uniform dimension of $R/P$, so that, by \cite[Propositions 7.24 and 7.8(c)]{GW}, there is an exact sequence of left $R$-modules
\begin{equation}\label{squeeze} 0 \rightarrow U^{\oplus s}\rightarrow R/P \rightarrow Y \rightarrow 0,
\end{equation}
where $Y$ is a torsion left $R/P$-module. In particular, $\mathrm{Ann}_{R}(Y) \supsetneqq P$, by \cite[Lemma 9.2]{GW}, so that $\mathrm{Kdim}(Y) < \mathrm{Kdim}(R/P) \leq r$, by \cite[Proposition 6.3.11(ii)]{MR}. This means that our induction hypothesis (on $r$) applies to $Y$. Let $\mathfrak{q}$ be an annihilator prime of $\,_{Z(R)}Y$. By INC Proposition \ref{INCetc}(iii), $\mathfrak{q} \supseteq \mathrm{Ann}_{R}(Y) \cap Z(R) \supsetneqq \mathfrak{p}$. So by induction and Theorem \ref{CMprops}(ii),
\begin{equation} \label{suck} j_R^{\ell}(Y) = \mathrm{inf}\{G_{Z(R)}(\mathfrak{q}, R) : \mathfrak{q} \textit{ an annihilator prime of } \,_{Z(R)}Y\} > t. \end{equation}
From (\ref{suck}) and the long exact sequence of $\mathrm{Ext}$ applied to (\ref{squeeze}) we obtain the exact sequence
$$  0 = \mathrm{Ext}^t(Y,R)\rightarrow  \mathrm{Ext}^t(R/P,R)\rightarrow  \mathrm{Ext}^t(U^{\oplus s},R). $$
Noting the definition (\ref{uniform}) of $t$, along with hypothesis (ii), this sequence shows that $j_R^{\ell}(U) \leq t.$ On the other hand, for all $i$ with $0 \leq i < t$, we have the exact sequence
\begin{equation} \label{see} \mathrm{Ext}^i(R/P,R)\rightarrow \mathrm{Ext}^i(U^{\oplus s},R)\rightarrow \mathrm{Ext}^{i+1}(Y,R). \end{equation}
Here, $i+1 \leq t < j_R^{\ell}(Y)$ by (\ref{suck}), so both outer terms in (\ref{see}) are 0, and hence $j_R^{\ell}(U) \geq t.$ Thus (\ref{uniform}) is proved.

Returning now to our arbitrary finitely generated $R$-module $X$ with $\mathrm{Kdim}(X) = r$, we first note that, by \cite[Theorem 9.6]{GW}, since $R$ is a finite module over its centre and thus FBN, $X$ has a finite chain of submodules
\begin{equation}\label{grow} 0 = X_0 \subset X_1 \subset \cdots \subset X_m = X,
\end{equation}
where, for $i = 1, \ldots , m$, $X_i/X_{i-1}$ is isomorphic to a uniform left ideal of a prime factor ring $R/P_i$ of $R$. Moreover, $\mathrm{Kdim}(R/P_i)=\mathrm{Kdim}(U_i)\leq r$ for all $i$, and the number $w$ of subfactors in (\ref{grow}) with Krull dimension precisely $r$ is easily checked to be an invariant of $X$. We shall prove the result for $X$ by induction on $w$, although observe that the case $w = 1$ is not yet proved.

Let $\mathfrak{i} = \mathrm{Ann}_{Z(R)}(X)$, and let $\mathfrak{q}_1, \ldots , \mathfrak{q}_c$ be the minimal primes over $\mathfrak{i}$. Let $\mathfrak{q}_1$ have minimal codimension amongst $\{\mathfrak{q}_j : 1 \leq j \leq c \}$, and set $\mathrm{ht}(\mathfrak{q}_1) = t$. By a straightforward exercise on the structure of finitely generated modules over commutative noetherian rings, $\mathrm{Ann}_{Z(R)}(X/\mathfrak{q}_jX) = \mathfrak{q}_j$ for all $j = 1, \ldots , c.$ Thus, building a chain (\ref{grow}) with the $R$-module $\mathfrak{q}_1X$ occurring as one of the terms, we easily see that we can assume that $P_m = \mathrm{Ann}_R(X/X_{m-1})$ satisfies $P_m \cap Z(R) = \mathfrak{q}_1.$ By (\ref{uniform}),
\begin{equation} \label{killer} j_R^{\ell}(X/X_{m-1}) = t. \end{equation}

The case $w = 1$ is now precisely the case where $\mathrm{Kdim}(X_{m-1}) < r$. In this case, by our induction on $r$ and our choice of $\mathfrak{q}_1$,
\begin{equation}\label{hip} j_R^{\ell}(X_{m-1}) = \mathrm{inf}\{G_{Z(R)}(\mathfrak{p}, R) : \mathfrak{p} \textit{ an annihilator prime of } \,_{Z(R)}X_{m-1}\} \geq t.\end{equation}
Then the exact sequence
\begin{equation} \label{blow}\mathrm{Ext}_R^{t-1}(X_{m-1},R)\rightarrow\mathrm{Ext}_R^t(X/X_{m-1},R)\rightarrow\mathrm{Ext}_R^t(X,R)\rightarrow\mathrm{Ext}_R^t(X_{m-1},R)
\end{equation}
has first term 0 and second term non-zero, by (\ref{hip}) and (\ref{uniform}). Thus $j_R^{\ell}(X) \leq t$; and the fact that $j_R^{\ell}(X) \geq t$ is clear from (\ref{blow}) with $t$ replaced by $i$, where $i < t$. Finally, the induction step for $w$ is covered by an exactly similar argument, where $X_{m-1}$ is handled by induction on $w$ rather than induction on $r$. This completes the proof of (iii).

{\rm(ii)}$\Leftrightarrow$ {\rm(iv)}: This follows from Rees's Change of Rings Theorem, \cite[Theorem 9.37]{R}. In more detail, for (iv)$\Rightarrow$(ii), the result is immediate from (\ref{change}) with $m=0$ and $I = \sum_i Rx_i.$ For the converse, assume (ii), and let $P$, $x_1, \ldots , x_t$ be as in (iv). Set $I = \sum_i Rx_i.$ Then, arguing on the left, $\mathrm{Ext}^t_R(R/P,R) \neq 0$ by (ii), so
\begin{equation}\label{lick}\mathrm{Hom}_R(R/P, R/I) \neq 0
\end{equation}
by (\ref{change}). But $R/I$ has an artinian quotient ring by Theorem \ref{CMprops}(vi) and (vii), so in particular all the (left) annihilator primes of $R/I$ are minimal over $I$. Hence (\ref{lick}) forces $P$ to be a left annihilator prime of $R/I$. A similar argument works on the right. This proves (iv).
\end{proof}

Recall that all quotients of $R$ of the form $R/\sum_i x_iR$, where $\{x_1, \ldots , x_t \}$ is a $Z(R)$-sequence, have artinian quotient rings $Q(R/\sum_i x_iR)$, by Theorem \ref{CMprops}(vi) and (vii). Therefore the condition on $R$ stated as Theorem \ref{gradequiv}(iv) is equivalent to requiring that, for every $Z(R)$-sequence $\{x_1, \ldots , x_t \}$ in $R$, every simple left [resp., right] $Q(R/\sum_i x_iR)$-module occurs as a left [resp., right] ideal of $Q(R/\sum_i x_iR)$. Artinian rings with this property are called \emph{(right and left) Kasch rings}, in honour of their definition and initial study in \cite{K}.

The following result, Corollary \ref{local}, shows that the property of being $Z(R)$-Macaulay grade symmetric exhibits a standard local-global condition. Observe that it follows from it that a $Z(R)$-Macaulay ring $R$ is grade symmetric if and only if all its factors by maximal $Z(R)$-sequnces are right and left Kasch rings.

\begin{corollary} \label{local} Let $R$ be a ring which is a finite module over its noetherian centre $Z(R)$.

 {\rm(i)} Let $P$ be a prime ideal of $R$, $\mathfrak{p} = P \cap R$. If $R$ is $Z(R)$-Macaulay and grade symmetric, then $R_{\mathfrak{p}}$ is $Z(R_{\mathfrak{p}})$-Macaulay and grade symmetric.

 {\rm(ii)} Suppose that, for every maximal ideal $\mathfrak{m}$ of $Z(R)$, $R_{\mathfrak{m}}$ is $Z(R_{\mathfrak{m}})$-Macaulay and grade symmetric, then $R$ is $Z(R)$-Macaulay and grade symmetric.
\end{corollary}
\begin{proof} (i) Suppose that $R$ is $Z(R)$-Macaulay and grade symmetric, and let $R$, $P$ and $\mathfrak{p}$ be as stated. Let $\widehat{Q}$ be a prime ideal of $R_{\mathfrak{p}}$, so $\widehat{Q} = Q_{\mathfrak{p}}$ for a prime $Q$ of $R$ with $Q \cap \{Z(R) \setminus \mathfrak{p}\} = \emptyset.$ Let $t := G_{Z(R)}(\mathfrak{q},R) = j_R^{\ell}(R/Q)$, the second equality by (i)$\Rightarrow$(ii) of Theorem \ref{gradequiv}. Now $\mathrm{Ext}_R^t(R/Q,R)$ is $R/Q$-torsion-free by Lemma \ref{muller}(ii). Localising at $\mathfrak{q}$, we deduce that
\begin{equation}\label{burn} j_{R_{\mathfrak{p}}}^{\ell}(R_{\mathfrak{p}}/\widehat{Q}) =  t  \leq
G_{Z(R_{\mathfrak{p}})}(\mathfrak{q}_{\mathfrak{p}}, R_{\mathfrak{p}}),\end{equation}
the second inequality by \cite[Lemma 18.1]{E}. Since $j_{R_{\mathfrak{p}}}^{\ell}(R_{\mathfrak{p}}/\widehat{Q}) \geq G_{Z(R_{\mathfrak{p}})}(\mathfrak{q}_{\mathfrak{p}}, R_{\mathfrak{p}})$ by Lemma \ref{grades}(i), equality holds in (\ref{burn}). That is, $R_{\mathfrak{p}}$ satisfies condition (ii) of Theorem \ref{gradequiv}. Since $R_{\mathfrak{p}}$ is $Z(R)_{\mathfrak{p}}$-Macaulay by Theorem \ref{CMprops}(iii), it is also grade symmetric by (ii)$\Rightarrow$(i) of Theorem \ref{gradequiv}.

(ii) Suppose the stated hypotheses hold. Then $R$ is $Z(R)$-Macaulay by Theorem \ref{CMprops}(iv). Therefore, by (ii) of Theorem \ref{gradequiv}, it is enough to prove that, given a prime ideal $P$ of $R$,
\begin{equation}\label{ghee}  j_R^{\ell}(R/P) = G_{Z(R)}(\mathfrak{p}, R). \end{equation}
So, let $P$ be a prime of $R$, and set $t := G_{Z(R)}(\mathfrak{p}, R)$. Lemma \ref{grades}(i) ensures that $j_R^{\ell}(R/P) \geq t.$ By \cite[Lemma 18.1]{E}, there is a maximal ideal $\mathfrak{m}$ of $Z(R)$ with $\mathfrak{p} \subseteq \mathfrak{m},$  such that $G_{Z(R_{\mathfrak{m}})}(\mathfrak{p}_{\mathfrak{m}}, R_{\mathfrak{m}}) = t.$ Hence, by (ii) of Theorem \ref{gradequiv} applied to $R_{\mathfrak{m}}$,
$$ 0 \neq \mathrm{Ext}_{R_{\mathfrak{m}}}^t(R_{\mathfrak{m}}/P_{\mathfrak{m}}
,R_{\mathfrak{m}}) \cong \mathrm{Ext}_R^t(R/P,R) \otimes R_{\mathfrak{m}}. $$
Thus $j^{\ell}_R(R/P) \leq t$, and (\ref{ghee}) is proved.
\end{proof}

We can now deduce that, for a grade symmetric centrally Macaulay ring $R$, the homological grade of finitely generated $R$-modules is entirely a function of their $Z(R)$-structure.

\begin{corollary} \label{carlisle} Let $R$ be a ring which is a finite module over its noetherian centre $Z(R)$, and let $X$ be a finitely generated left $R$-module. Suppose that $R$ is $Z(R)$-Macaulay and grade symmetric. Then
\begin{align*} j_R^{\ell}(X) &= j_R^{\ell}(R/\mathrm{Ann}_R(X)) \\
&= \mathrm{inf}\{G_{Z(R)}(\mathfrak{p}, R): P \textit{ minimal prime over } \mathrm{Ann}_R(X) \} \\
&= \mathrm{inf}\{\mathrm{ht}(\mathfrak{p}): \mathfrak{p} \textit{ minimal prime over } \mathrm{Ann}_{Z(R)}(X) \}.
\end{align*}
\end{corollary}
\begin{proof} Let $R$ and $X$ be as stated. We start by showing that the right hand side of (\ref{value}), which we set as $t$, is equal to the second and third displayed lines in the corollary. Set $I = \mathrm{Ann}_R(X)$ and $\mathfrak{i} = I \cap Z(R).$ Let $\mathfrak{q}$ be a minimal prime over $\mathfrak{i}$, and let $\mathfrak{p}_1, \dots , \mathfrak{p}_r$ be the annihilator primes of $\,_{Z(R)}X.$ By basic properties of finitely generated $Z(R)$-modules, in particular the Artin-Rees lemma \cite[Theorem 13.3]{GW}, some finite product of the primes $\mathfrak{p}_1, \dots , \mathfrak{p}_r$ is contained in $\mathfrak{i}$. Hence, there exists $j$ for which $\mathfrak{p}_j \subseteq \mathfrak{q}$. By minimality of $\mathfrak{q}$ we deduce that $\mathfrak{p}_j = \mathfrak{q}$. Thus every minimal prime over $\mathfrak{i}$ is an annihilator prime of $\,_{Z(R)}X$. Since $G(\mathfrak{p},R) \geq G(\mathfrak{q},R)$ if $\mathfrak{p} \supseteq \mathfrak{q}$, $t$ must be attained as $G(\mathfrak{q},R)$ for a prime $\mathfrak{q}$ minimal over $\mathfrak{i}$; that is, $t$ equals the third displayed line.

Let $\mathfrak{q}$ be any minimal prime over $\mathfrak{i}$. By the argument above, $\mathfrak{q} = \mathrm{Ann}_{Z(R)}(Y)$ for some $Z(R)$-submodule $Y$ of $X$, and we can take $Y$ to be an $R$-module. The annihilator primes $Q_1, \ldots , Q_m$ in $R$ of $Y$ clearly all contain $\mathfrak{q}R$. Suppose that $\mathfrak{j} := \cap_j Q_j \cap Z(R) \varsupsetneqq \mathfrak{q}.$ Then $A := \mathrm{Ann}_Y(\mathfrak{j}R)$ is an essential $R$-submodule of $Y$, and its annihilator $\mathfrak{j}R$, being centrally generated, satisfies the Artin-Rees property, \cite[Theorem 13.3]{MR}. Hence, some power $\mathfrak{j}^w R$ of it annihilates $Y$, which is impossible since $\mathfrak{j}^w \nsubseteq \mathfrak{q}$. Thus $\cap_j Q_j \cap Z(R) = \mathfrak{q},$ and so there exists $j$ with $Q_j \cap Z(R) = \mathfrak{q}$. Notice finally that $I \subseteq Q_j$; and if $Q_j$ were not minimal over $I$, say $Q_j \varsupsetneqq  P \supseteq I$, then INC, Proposition \ref{INCetc}(iii), ensures $\mathfrak{q} \varsupsetneqq \mathfrak{p} \supseteq \mathfrak{i}$, contradicting minimality of $\mathfrak{q}$ over $\mathfrak{i}$. Therefore $Q_j$ is minimal over $I$, and $t$ equals the second displayed line.

 Now (\ref{value}) applied first to the module $X$, and second to the module $R/\mathrm{Ann}_R(X)$ shows that the homological grade of both of these modules is given by $$\mathrm{inf}\{G_{Z(R)}(\mathfrak{p}, R) : \mathfrak{p} \textit{ an annihilator prime of } \,_{Z(R)}X\}.$$ Combined with the equalities proved above, this proves the corollary.
\end{proof}

It is a routine exercise, which we omit, to deduce the following consequence of Corollary \ref{carlisle}. For the definition of \emph{finitely partitive}, see for example \cite[8.3.17]{MR}. A similar result was obtained for Auslander-Gorenstein noetherian rings by Levasseur in \cite[Proposition 4.5]{Lev}.

\begin{corollary} \label{dim} Let $R$ be a ring which is a finite module over its noetherian centre $Z(R)$, with $R$ grade symmetric, $Z(R)$-Macaulay and with $\mathrm{Kdim}R = t < \infty$. Then  $\delta := t - j_R$ is an exact finitely partitive dimension function on the category of finitely generated $R$-modules.
\end{corollary}

\section{Homological hierarchy}\label{homhier}

\subsection{Injective and homological homogeneity}\label{hom}

\begin{definition}\label{injhom}Let $R$ be a noetherian ring which is a finite module over its centre $Z(R)$.

{\rm (i)} Let $I$ be an ideal of $R$. The \emph{(right) upper grade} of $I$ is $\mathrm{r.u.gr}(I) = \mathrm{max}\{ i : \mathrm{Ext}_R^i(R/I_{|R}, R_{|R}) \neq 0 \}$, (and is $\infty$ if no such maximum exists).

{\rm(ii)} $R$ is \emph{(right) injectively homogeneous} if $R$ has finite (right) injective dimension, and $\mathrm{r.u.gr.}(M) = \mathrm{r.u.gr.}(N)$ for all maximal ideals $M$ and $N$ of $R$ for which $M\cap Z(R) = N\cap Z(R)$.

{\rm(iii)} $R$ is \emph{(right) homologically homogeneous} if it is right injectively homogeneous and has finite global dimension.
\end{definition}
\begin{remarks}\label{injhomrems} (i) The original definitions \cite{BrH}, \cite {BrH2} of homological and injective homogeneity were more general than the above, in that $R$ was allowed to be merely integral over its centre. We have been more restrictive here for the sake of brevity.

(ii) The adjective ``right'' can be dropped from Definitions \ref{injhom}(ii) and (iii). This was shown in \cite[Corollaries 4.4 and 6.6]{BrH2} and is incorporated into the next theorem and its corollary.

(iii) Let $R$ be a commutative noetherian ring which is a finitely generated module over the central subring $C$. Dao, Faber and Ingalls, \cite[Definition 2.1]{DFI}, following Iyama and Wemyss, \cite[Definition 1.6]{IW}, call a ring $R$ a \emph{non-singular $C$-order} if $R$ is a maximal Cohen-Macaulay $C$-module and $\mathrm{gl.dim.}(R_{\mathfrak{p}}) = \mathrm{Kdim}(C_{\mathfrak{p}})$ for all primes $\mathfrak{p}$ of $C$. In doing so both sets of authors were following Van den Bergh \cite[$\S$3]{VdB}, who called such rings \emph{homologically homogeneous}, and noted that this usage coincided with Definition \ref{injhom}(iii), provided $C$ is equicodimensional. A proof of this claim is provided at \cite[Proposition 2.14]{DFI}. The interested reader should consult these papers, and also \cite{L}, for further background of the connections with noncommutative resolution of singularities.

\end{remarks}

\begin{theorem}\label{inj}

Let $R$ be a ring which is a finite module over its noetherian centre $Z(R)$. Then the following are equivalent.

{\rm(i)} $R$ is left injectively homogeneous.

{\rm(ii)} $\;$(1) $R$ is $Z(R)$-Macaualay;

\indent $\qquad$(2)  $R$ is grade symmetric; and

\indent $\qquad$(3) $R$ has finite left injective dimension.

{\rm(iii)} $\;$(1) $R$ is $Z(R)$-Macaualay;

\indent $\qquad$(2)  $R$ is grade symmetric; and

\indent $\qquad$(3) $R$ has finite right injective dimension.

{\rm(iv)} $R$ is right injectively homogeneous.
\end{theorem}

\begin{proof} Since (i) is equivslent to (iv) by \cite[Corollary 4.4]{BrH2}, it is enough to prove that (i) and (ii) are equivalent, since the equivalence of (iii) and (iv) will follow by a parallel argument.

{\rm(i)}$\Rightarrow${\rm(ii)}: Suppose that $R$ is left injectively homogeneous. Then (1) is assured by (i)$\Rightarrow$(ii) of \cite[Theorem 3.4]{BrH2}. For (2), since $R$ is $Z(R)$-Macaulay, it is enough to prove (ii) of Theorem \ref{gradequiv}. So, let $P$ be a prime ideal of $R$. Then, as required, $$j_R^{\ell}(R/P) = \mathrm{ht}(P) = G_{Z(R)}(\mathfrak{p},R),$$ the first equality by \cite[Theorem 5.3]{BrH2} and the second by Theorem \ref{CMprops}(ii). Finally, (3) is part of the definition of the injectively homogeneous property.

{\rm(ii)}$\Rightarrow${\rm(i)}: Suppose that $R$ satisfies properties (1), (2) and (3), with $\mathrm{l.inj.dim}(R) := n < \infty$. Consider first the special case where $Z(R)$ is local, with maximal ideal $\mathfrak{m}$. Let $x_1, \ldots , x_t$ be a maximal $Z(R)$-sequence in $\mathfrak{m}$ on $R$, and set $I = \sum_i Rx_i$. Notice that
\begin{equation} \label{bound} t \leq n \quad \textit{ and } \quad \mathrm{injdim}(R/I) = n-t,\end{equation}
by Rees's Change of Rings theorem \cite[Theorem 9.37]{R}, (see (\ref{change})). By (i)$\Rightarrow$(ii) of Theorem \ref{gradequiv}, every simple left $R$-module $X$ has
\begin{equation} \label{finch} j_R^{\ell}(X) = t;\end{equation}
hence, by (\ref{change}) again, each of these simple modules occurs in the left socle of $R/I$. By \cite[Lemma 3.1]{BrH2}, there is a simple left $R$-module $Y$ with $\mathrm{Ext}_R^n(Y,R) \neq 0.$ Applying (\ref{change}) again shows that $\mathrm{Ext}_{R/I}^{n-t}(Y,R/I) \neq 0.$ To the exact sequence of left $R/I$-modules
$$ 0 \rightarrow Y \rightarrow R/I \rightarrow (R/I)/Y \rightarrow 0 $$
we can apply the functor $\mathrm{Hom}_{R/I}(-, R/I)$ to get the exact sequence
\begin{equation}\label{fine} \mathrm{Ext}_{R/I}^{n-t}(R/I,R/I) \rightarrow \mathrm{Ext}_{R/I}^{n-t}((R/I)/Y,R/I)\rightarrow \mathrm{Ext}_{R/I}^{n-t+1}(Y,R/I). \end{equation}
The last term of (\ref{fine}) is 0 by (\ref{bound}), so the first term must be non-zero. Hence, $n = t$.

It follows from this and (\ref{finch}) that $\mathrm{l.u.gr.}(M) = t$ for every maximal ideal $M$ of $R$. Therefore $R$ is injectively homogeneous.

Now drop the hypothesis that $Z(R)$ is local. For every maximal ideal $\mathfrak{m}$ of $Z(R)$, $R_{\mathfrak{m}}$ satisfies hypotheses (1), (2) and (3), with $\mathrm{l.inj.dim}(R_{\mathfrak{m}}) \leq n$, by Corollary \ref{local}(i) and the behaviour of $\mathrm{Ext}$ under central localisation, \cite[11.58]{R}. So $R_{\mathfrak{m}}$ is left injectively homogeneous by the above argument. Therefore $R$ is left injectively homogeneous by the local-global property of injective homogeneity, \cite[Lemma 3.3]{BrH2}.
\end{proof}

\begin{corollary}\label{hom}

Let $R$ be a ring which is a finite module over its noetherian centre $Z(R)$. Then the following are equivalent.

{\rm(i)} $R$ is homologically homogeneous.

{\rm(ii)} $\;$(1) $R$ is $Z(R)$-Macaualay;

$\qquad$(2)  $R$ is grade symmetric; and

$\qquad$(3) $R$ has finite global dimension.

\end{corollary}
\begin{proof}{\rm(i)}$\Rightarrow${\rm(ii)}:Since homologically homogeneous rings are injectively homogeneous with finite global dimension, this is immediate from the theorem.

{\rm(ii)}$\Rightarrow${\rm(i)}: Suppose that $R$ satisfies (1), (2) and (3). By Theorem \ref{inj}, $R$ is injectively homogeneous. Since $R$ has finite global dimension, it is homologically homogeneous by \cite[Theorem 6.5]{BrH}.
\end{proof}

\section{Speculations beyond PI} \label{dream} Throughout this section $k$ is an arbitrary field. Given integers $i$ and $n$ and a complex $B$, the \emph{shift operator} $[n]$ is defined by $B[n]_i := B_{n+i}.$ We discuss here what might be the appropriate extension of the definition of a Cohen-Macaulay ring beyond the commutative setting.

For the definition and basic properties of rigid dualizing complexes, see \cite{VdB2}, \cite{YZ}. So far as we are aware, it remains possible that \emph{every} affine noetherian $k$-algebra of finite GK-dimension has a rigid dualizing complex. When a noetherian ring $A$ does possess a rigid dualizing complex $R$, then $R$ is unique up to a unique isomorphism, \cite[Proposition 8.2(1)]{VdB2}, \cite[Theorem 5.2]{Y}. Here is the first of two ``generalised Cohen-Macaulay conditions" which will feature in this discussion.

\begin{definition} \label{AS-CM} (Van den Bergh, \cite[page 674]{VdB2}) A noetherian ring $A$ with a rigid dualizing complex $R$ is \emph{AS-Cohen Macaulay} if the complex $R$ is concentrated in one degree.
\end{definition}

Suppose that $A$ is an algebra of a type which has featured in previous sections of this paper, specifically, an indecomposible affine noetherian $k$-algebra which is $Z(A)$-Macaulay of Gel'fand-Kirillov dimension $n$. Then $A$ is equicodimensional by Theorem 3.4. Take a central polynomial subalgebra $C$ in $n$ variables over which $A$ is a finite module, as afforded by Noether normalisation. Then $A$ is a free $C$-module of finite rank by Theorem 3.7(ii). Hence, since $C[n]$ is the rigid dualizing complex of $C$, it follows from \cite[Example 3.11]{YZ}, \cite[Proposition 5.7]{Y} that
$$ R \quad := \quad \mathrm{RHom}_C(A, C[n]) $$
is the rigid dualizing complex of $A$. That is, we have obtained the following well-known fact, which for convenience we state as

\begin{proposition} \label{rigidexist} Let $A$ be an indecomposible affine noetherian $k$-algebra which is $Z(A)$-Macaulay of Gel'fand-Kirillov dimension $n$. Then $A$ is AS-Cohen Macaulay, with its rigid dualizing complex concentrated in the single degree $-n$.
\end{proposition}

Note, however, that there is no grade symmetry component to the hypothesis or conclusion of the proposition, so something else is needed for a condition which would fit into a hypothetical generalised formulation of the homological hierarchy of $\S 5$.

In \cite[Definition 2.1]{YZ}, Yekutieli and Zhang introduce the concept of an \emph{Auslander} dualizing complex $R$ of an algebra $A$. By this they meant that, for all finitely generated $A$-modules $M$, the familiar Auslander property should be satisfied, but for the $A$-modules $\mathrm{Ext}^i_A (M, R)$, $(i \in \mathbb{Z},$ rather than the modules $\mathrm{Ext}^i_A (M, A)$, $(i \in \mathbb{N}$. They show that under a range of hypotheses on the algebra $A$, the rigid dualizing complex of $A$ is Auslander; and they ask \cite[Question 0.11]{YZ}, restated here for convenience as
\begin{question} \label{rigidAus} Is every rigid dualizing complex Auslander?
\end{question}
Suppose that $A$ is a noetherian $k$-algebra with an Auslander rigid dualizing complex $R$. Define the $R$-\emph{grade} of a finitely generated $A$-module $M$ to be
$$ j_R (M) \quad := \quad \mathrm{inf}\{i : \mathrm{Ext}^i_A (M, R) \neq 0 \} \in \mathbb{Z} \cup \infty. $$
Then Yekutieli and Zhang define the \emph{canonical dimension} of $M$ to be
\begin{equation} \label{flip}\mathrm{Cdim}_R (M) \quad := \quad - j_R (M) \quad \in \quad \mathbb{Z} \cup - \infty,
\end{equation}
and show in \cite[Theorem 2.10]{YZ} that, with these hypotheses on $A$,
$$\mathrm{Cdim}_R \textit{ is an exact finitely partitive dimension function}.$$
They then ask \cite[Question 3.15]{YZ}:

\begin{question} \label{Cdimsym}Let $A$  be a noetherian $k$-algebra with a rigid dualizing complex $R$. Is $\mathrm{Cdim_R}$ symmetric on central $A$-bimodules?
\end{question}

In \cite{YZ}, Yekutieli and Zhang exhibit a number of cases where a positive answer to Question \ref{Cdimsym} is known. One such setting arises from the following second ``generalised Cohen-Macaulay condition", slightly adapted here from the original by requiring rigidity:

\begin{definition}\label{flop} (Yekutieli, Zhang, \cite[Definition 2.24]{YZ}) Let $A$ be a noetherian $k$-algebra of finite Gel'fand-Kirillov dimension with an Auslander rigid dualizing complex. Then $A$ is \emph{dualizing GK-Macaulay} if there is an integer $c$ such that
\begin{equation}\label{gotcha} \mathrm{GKdim} (M) - \mathrm{Cdim}_R (M) \quad = \quad c \end{equation}
for all non-zero finitely generated left or right $A$-modules $M$.
\end{definition}

Yekutieli and Zhang use the term \emph{GKdim-Macaulay} for the above; we have introduced new terminology to avoid confusion with the more familiar GK-Macaulay condition recalled here at Definition \ref{graderex}(ii). Since the Gel'fand-Kirillov dimension is symmetric by \cite[Corollary 5.4]{KL}, $\mathrm{Cdim}_R$ \emph{is} symmetric when $A$ is dualizing GK-Macaulay.

Observe that being dualizing GK-Macaulay is a rather weak condition on an algebra. For example, we have:

\begin{theorem}\label{dualGK} (Yekutieli, Zhang) Let $A$ be an affine noetherian $k$-algebra which is a finite module over its centre. Then $A$ is dualizing GK-Macaulay.
\end{theorem}
\begin{proof} This follows immediately from \cite[Corollary 6.9 and Example 6.14]{YZ}.
\end{proof}

Combining this circle of ideas in our laboratory setting of centrally Macaulay rings yields the following suggestive result. Recall that \emph{grade symmetry} was defined at Definition \ref{symdef}(iii).

\begin{proposition}\label{hint} Let $A$ be an affine indecomposible $Z(A)$-Macaulay $k$-algebra of Gel'fand-Kirillov dimension $n$. Let $R$ be the rigid dualizing complex of $A$. Then the following are equivalent:

{\rm(i)} $A$ is grade symmetric;

{\rm(ii)}the map $j - j_R : \mathrm{mod}A \longrightarrow \mathbb{Z}$ is constant.

\end{proposition}
\begin{proof} By Theorem \ref{dualGK}, $A$ is dualizing GK-Macaulay.

(ii)$\Rightarrow$(i):  As noted above, immediately after Definition \ref{flop}, $\mathrm{Cdim}_R$ is symmetric since $A$ is dualizing GK-Macaulay.  By (\ref{flip}) this symmetry is also satisfied by $j_R$, so (i) follows from (ii).

(i)$\Rightarrow$(ii): Assume that $A$ is grade symmetric. Then $A$ is GK-Macaulay by Theorem \ref{opus}(i)$\Leftrightarrow$(ii). That is, for all finitely generated $A$-modules $M$,
$$ j(M) = n - \mathrm{GKdim}(M). $$
On the other hand, by (\ref{gotcha}) and (\ref{flip}) there exists $c \in \mathbb{Z}$ such that
$$ j_R(M) = c - \mathrm{GKdim}(M). $$
Thus (ii) follows.
\end{proof}

Guided therefore by Definition \ref{AS-CM} and by Proposition \ref{hint}, we might propose that
\begin{align*} \textit{the noetherian } k-\textit{algebra } A & \textit{ should} \textit{ be regarded as Cohen-Macaulay }\\
 &\Leftrightarrow \\
  A \textit{ has a rigid Auslander dualizing }& \textit{complex } R \textit{ concentrated in one degree,}\\ \textit{ and } j - j_R &\textit{ is constant.} \end{align*}
Of course, this suggestion immediately begs another question: how to determine whether a given algebra satisfies the suggested conditions.


 \end{document}